\documentclass[12pt]{article}
\usepackage{amssymb,amsfonts,amsthm,amsmath,mathrsfs,xspace,hyperref}
\usepackage{graphicx}
\usepackage[francais,english]{babel}
\usepackage{inputenc}
\usepackage[T1]{fontenc}
\usepackage{authblk}
\usepackage[centering]{geometry}
\usepackage{csquotes}
\usepackage{color}
\usepackage{enumitem}

   \newcommand{\RR}{\ensuremath{\mathbf{R}}}%
      \newcommand{\CC}{\ensuremath{\mathbf{C}}}%
    \newcommand{\ZZ}{\ensuremath{\mathbf{Z}}}%
	\newcommand{\F}{\ensuremath{\mathcal{F}}}%
\newcommand{\QZ}{\ensuremath{\mathrm{QZ}}}%

                \newcommand{\Ccal}{\ensuremath{\mathcal{C}}}%
	\newcommand{\fix}{\ensuremath{\textrm{Fix}}}%
  \newcommand{\sub}{\ensuremath{\mathbf{Sub}}}%
    \newcommand{\aut}{\ensuremath{\operatorname{Aut}}}%
		\newcommand{\comm}{\ensuremath{\operatorname{Comm}}}%
        \newcommand{\Aut}{\ensuremath{\operatorname{Aut}}}%
        \newcommand{\SL}{\ensuremath{\operatorname{SL}}}%

\newcommand{\Ker}{\mathrm{Ker}} 
\newcommand{\Mon}{\mathrm{Mon}}

\theoremstyle{definition}
  \newtheorem{defin}{Definition}[section]

  \newtheorem{question}{Question}[section]

\theoremstyle{plain}
  \newtheorem{thm}[defin]{Theorem}
  \newtheorem{main thm}{Theorem}
  \newtheorem{prop}[defin]{Proposition}
    \newtheorem{prop-def}[defin]{Proposition-Definition}
    \newtheorem{thmintro}{Theorem}
    
      \newtheorem{corintro}[thmintro]{Corollary}

  \newtheorem{cor}[defin]{Corollary}
\newtheorem{lem}[defin]{Lemma}

\theoremstyle{remark}
  \newtheorem{rmq}[defin]{Remark}

\title{Bounding the covolume of lattices in products}
\author[1]{Pierre-Emmanuel Caprace\thanks{F.R.S.-FNRS Senior Research Associate. \texttt{pe.caprace@uclouvain.be}}}
\author[2]{Adrien Le Boudec\thanks{CNRS Researcher. This work was partially carried out when the second author was F.R.S.-FNRS Post-Doctoral Researcher at UCLouvain. \texttt{adrien.le-boudec@ens-lyon.fr} This work was partially supported by ANR-14-CE25-0004 GAMME.}}
\affil[1]{UCLouvain, 1348 Louvain-la-Neuve, Belgium}
\affil[2]{UMPA - ENS Lyon, France}

\date{October 23, 2019}

\begin{document}
	
\maketitle

\begin{abstract}
We study lattices in a product $G = G_1 \times \dots \times G_n$ of non-discrete, compactly generated, totally disconnected locally compact (tdlc) groups. We assume that each factor is \textbf{quasi just-non-compact}, meaning that $G_i$ is non-compact and every closed normal subgroup of $G_i$ is discrete or cocompact (e.g. $G_i$ is topologically simple).

We show that the set of discrete subgroups of $G$ containing a fixed cocompact lattice $\Gamma$ with dense projections is finite. The same result holds if $\Gamma$ is non-uniform, provided $G$ has Kazhdan's property (T). We show that for any compact subset $K \subset G$, the collection of discrete subgroups $\Gamma \leq G$ with $G = \Gamma K$ and dense projections is uniformly discrete, hence of covolume bounded away from~$0$. When the ambient group $G$ is compactly presented, we show in addition that the collection of those lattices falls into finitely many $\Aut(G)$-orbits. As an application, we establish finiteness results for discrete groups acting on products of locally finite graphs with semiprimitive local action on each factor. 

We also present several intermediate results of independent interest. Notably it is shown that if a non-discrete, compactly generated quasi just-non-compact tdlc group $G$ is a Chabauty limit of discrete subgroups, then some compact open subgroup of $G$ is an infinitely generated pro-$p$ group for some prime $p$. It is also shown that in any Kazhdan group with discrete amenable radical, the  lattices form an open subset of the Chabauty space of closed subgroups.
\end{abstract}

\setcounter{tocdepth}{2}

\tableofcontents

\section{Introduction}

\subsection{Covolume bounds}

A classical result of  H.~C.~Wang \cite{Wang} ensures in a connected semisimple Lie group $G$ without compact factor, the  collection of discrete subgroups  containing a given lattice $\Gamma \leq G$ is finite. Soon afterwards, an important closely related result   was established by Kazhdan--Margulis~\cite{Kaz-Marg-thm}, who proved that the set of  covolumes of all lattices in $G$ is bounded   below by a positive constant.   H.~C.~Wang~\cite{Wang72} subsequently used the Kazhdan--Margulis theorem in combination with local rigidity to establish that if $G$ has no factor locally isomorphic to $\SL_2(\RR)$ or $\SL_2(\CC)$, then for every $v>0$, the set of conjugacy classes of lattices in $G$ of covolume~$\leq v$ is finite (see also \cite[Theorem 13.4]{Gel-homotopy} for the case of irreducible lattices when the group $G$ itself is not locally isomorphic to $\SL_2(\RR)$ or $\SL_2(\CC)$).   

Bass--Kulkarni showed in  \cite[Theorem~7.1]{BK90} that none of those results   holds when $G$ is the full automorphism group of the $d$-regular tree $T_d$, with $d \geq 5$, even if one restricts to cocompact lattices.  In particular, since $G= \mathrm{Aut}(T_d)$  is compactly generated and has a simple open subgroup of index~$2$ (see \cite{Tits_arbre}), none of the results above can be expected to hold for cocompact lattices in compactly generated, topologically  simple, locally compact groups in general. We recall that a locally compact group is \textbf{topologically simple} if it is non-trivial and the only closed normal subgroups are the trivial ones. 	

In this paper, we consider  cocompact lattices with dense projections in products of non-discrete, compactly generated, topologically simple, locally compact groups. Our goal   is to show that, in  this situation,   similar phenomena as in the case of semisimple Lie groups occur.  Our results are actually valid for lattices in products of locally compact groups that satisfy  a  condition that is weaker than simplicity. In order to define it, we recall that a locally compact group $G$ is called \textbf{just-non-compact} if it is non-compact and every closed normal subgroup is trivial or cocompact. We say that $G$ is \textbf{quasi just-non-compact} if it is non-compact and every closed normal subgroup is discrete or cocompact. We note that this definition is meaningful only in the realm of non-discrete groups.  As first observed by Burger-Mozes, quasi just-non-compact groups appear naturally in the context of automorphism groups of connected graphs with quasi-primitive local action \cite[Proposition~1.2.1]{BuMo1}.

Our first main result is the following close relative of the aforementioned theorem of H.~C.~Wang \cite{Wang}.  Throughout the paper, the abbreviation \textbf{tdlc} stands for \emph{totally disconnected locally compact}. 

\begin{thmintro}\label{thm:Wang1}
	Let $G_1,\ldots,G_n$ be non-discrete, compactly generated, quasi just-non-compact tdlc groups and let   $ \Gamma \leq G = G_1 \times \dots \times G_n$ be a   lattice such that the projection $p_i(\Gamma)$ is dense in $G_i$ for all $i$. Assume that at least one of the following conditions holds. 
	\begin{enumerate}[label=(\arabic*)]
		\item $\Gamma$ is cocompact. 
		
		\item $G$ has Kazhdan's property (T). 
	\end{enumerate}
	Then the set of discrete subgroups of $G$ containing $\Gamma$ is finite. 
\end{thmintro}

Theorem~\ref{thm:Wang1} implies in particular that any lattice $\Gamma$ as in the theorem is contained in a maximal lattice (and hence that maximal lattices exist). 

\begin{rmq}
In Theorem~\ref{thm:Wang1} as well as in the other statements of this introduction, we assume that the ambient group $G$ is totally disconnected. That restriction is rather mild, since the presence of a connected simple factor (or more generally, a quasi just-non-compact almost connected factor without non-trivial  connected abelian normal subgroup) implies much stronger structural constraints on the lattice $\Gamma$ and on the other  factors of the ambient product group, in view of the arithmeticity results from \cite[Theorem 5.18]{CaMo-discrete} and \cite[Theorem 1.5]{BFS-adelic}.
\end{rmq}

A version of the Wang finiteness theorem is established by Burger--Mozes in~\cite[Theorem~1.1]{BuMo_Wang} for cocompact lattices with dense projections in certain automorphism groups of trees with quasi-primitive local action. When specified to this setting, Theorem~\ref{thm:Wang1} allows one to recover and generalize their result (see Section~\ref{subsec:LocalAction} below).

In \cite[Theorem~1.8]{GelanderLevit}, Gelander--Levit provide sufficient conditions on a set of discrete subgroups of a locally compact group $G$ all containing a fixed finitely generated lattice $\Gamma$, to be finite. Their conditions are not satisfied \emph{a priori} under the hypotheses of Theorem~\ref{thm:Wang1}. However, the proof of the latter elaborates on some of the ideas developed in \cite{GelanderLevit}. Other ingredients are presented below. A remarkable feature of Theorem~\ref{thm:Wang1}, and also Theorems \ref{thmintro:KazMar} and~\ref{thmintro:Wang2} below, is that their proofs rely in an essential way on a combination of results from the recent structure theory of tdlc groups developed in \cite{CaMo-decomp,CRW-part2,CRW_dense} together with various results and ideas from finite group theory.

\bigskip

Given a compact subset $K \subseteq G$ of a locally compact group $G$, a subgroup $H$ is called \textbf{$K$-cocompact} if $G = HK$. The following result implies the existence of a positive lower bound on the set of covolumes of all  $K$-cocompact   lattices with dense projections.

\begin{thmintro}[See Theorem~\ref{thm:KazMar:Irr}] \label{thmintro:KazMar}
	Let $G_1,\ldots,G_n$ be non-discrete, compactly generated, quasi just-non-compact tdlc groups. For every compact subset $K \subset G = G_1 \times \dots \times G_n$, there  exists an identity neighbourhood $V_K$ such that $V_K \cap \Gamma = \{1\}$ for every $K$-cocompact discrete subgroup $\Gamma \leq G$ with $p_i(\Gamma)$ dense in $G_i$ for all $i$.  
\end{thmintro}

Given a compactly generated tdlc group $G$ and a compact open subgroup $U$, any $K$-cocompact subgroup acts with at most $r$ orbits on $G/U$, where $r = |KU/U|$. Conversely, for every $r >0$ there is a compact subset $K \subset G$ such that any subgroup of $G$ acting with at most $r$ orbits on $G/U$ is $K$-cocompact (see Lemma~\ref{lem-rU-cocom-uniform} below). 

For a discrete subgroup, the condition of $K$-cocompactness may be viewed as an upper bound on the covolume. However, that condition is generally strictly stronger than being cocompact and of covolume bounded above: indeed, Bass--Kulkarni have shown in  \cite[Theorem~7.1(b),7.19--7.20]{BK90} that for $d \geq 5$, the group $\Aut(T_d)$ contains an infinite family $(\Gamma_k)$ of cocompact lattices of constant covolume, and such that the number of vertex orbits of $\Gamma_k$ tends to infinity with $k$. In particular there does not exist any compact subset $K \subset \Aut(T_d)$ such that $\Gamma_k$ is $K$-cocompact for all $k$. 

We do not know whether there could exist a neighbourhood of the identity as in the conclusion of Theorem~\ref{thmintro:KazMar} that is actually independent of $K$ (Question \ref{quest-uniform-neighb} below). See \cite[Conjecture (B), IX.4.21]{Margulis} for a related conjecture in the case of semi-simple groups.

Using Serre's covolume formula, Theorem~\ref{thmintro:KazMar} implies that the set of covolumes of $K$-cocompact lattices with dense projections in $G$, is finite  (see Theorem~\ref{thm:KazMar:Irr} below). As mentioned above, in the classical case of semisimple Lie groups, the combination of the Kazhdan--Margulis theorem with local rigidity of lattices yields a much stronger finiteness statement, due to H.~C.~Wang~\cite{Wang72}.  A very general version of the local rigidity of cocompact lattices has recently been established by Gelander--Levit~\cite{GelanderLevit}.  Using their results, we establish the following. 

\begin{thmintro}[See Theorem~\ref{thm-finite-A-orbits}] \label{thmintro:Wang2} 
	Let $G_1,\ldots,G_n$ be non-discrete, compactly presented, quasi just-non-compact tdlc groups.  For every compact subset $K \subset G = G_1 \times \dots \times G_n$, the set of  $K$-cocompact discrete subgroups $\Gamma \leq G$ with $p_i(\Gamma)$ dense in $G_i$ for all $i$,  is contained in a union of finitely many $\aut(G)$-orbits.
\end{thmintro}

The hypothesis of compact presentability of the factors is needed to invoke the suitable local rigidity results from  \cite{GelanderLevit}.

\subsection{Discrete groups acting on product of graphs} \label{subsec:LocalAction}

A natural setting in which the previous results find applications is the study of lattices in products of trees, and more generally products of automorphism groups of graphs, with restricted local action. The investigation of such lattices was  pioneered by Burger--Mozes~\cite{BuMo1}, \cite{BuMo2}. 

Let $X$ be a connected graph. Given  $G\leq \Aut(X)$, a vertex $x \in VX$ and an integer $\ell \geq 0$, we denote by $G_x^{[\ell]}$ the pointwise stabilizer of the $\ell$-ball around $x$. Thus $G_x^{[0]} =G_x$ is simply the stabilizer of $x$. We also denote by $X(x)$ the vertices at distance~$1$ from $x$, and by $G_x^{X(x)}$ the permutation group induced by the action of the stabilizer $G_x$ on $X(x)$. As an abstract group, it is isomorphic to $G_x/G_x^{[1]}$. We call $G_x^{X(x)}$ the \textbf{local action} of $G$ at $x$.

We recall that a permutation group $L$ of a set $\Omega$ is \textbf{primitive} if the only $L$-invariant partitions of $\Omega$ are the trivial ones. The group $L$ is \textbf{quasi-primitive} if it is transitive and  the only intransitive normal subgroup is   trivial, and \textbf{semiprimitive} if it is transitive and every intransitive normal subgroup  acts freely. Primitive groups are quasi-primitive, and quasi-primitive groups are semiprimitive.

As observed by Burger--Mozes~\cite[Proposition~1.2.1]{BuMo1}, quasi just-non-compact groups appear naturally in the context of automorphism groups of connected graphs with quasi-primitive local actions. More generally, given a connected locally finite graph $X$, any closed subgroup $G \leq \Aut(X)$ with semiprimitive local action is quasi just-non-compact, see  Proposition~\ref{prop-BM-loc-semiprim} below. Therefore, Theorems~\ref{thm:Wang1},~\ref{thmintro:KazMar} and~\ref{thmintro:Wang2} can be applied to cocompact lattices with dense projections in a product of automorphism groups of connected graphs with semiprimitive local actions. In particular, Theorems~\ref{thm:Wang1},~\ref{thmintro:KazMar} and~\ref{thmintro:Wang2} have the following direct consequence. 

\begin{corintro}\label{corintro:graphs}
Let $X_1, \dots, X_n$ be connected locally finite graphs and for each $i$, let $G_i \leq \Aut(X_i)$ be a non-discrete closed subgroup with semiprimitive local action. Let $G = G_1 \times \dots \times G_n$. 
\begin{enumerate}[label=(\Alph*)]
\item \label{it:GraphsA} 
For every cocompact lattice $\Gamma \leq G$ such that $p_i(\Gamma)$ is dense in $G_i$ for all $i$, the set of discrete subgroups of $G$ containing $\Gamma$ is finite. 

\item \label{it:GraphsB} 
For each $r>0$, there exists a constant $c = c(r)$ such that for every cocompact lattice $\Gamma \leq G$ with at most $r$ orbits on $\prod VX_i$ and such that $p_i(\Gamma)$ is dense in $G_i$ for all $i$, we have  $\Gamma_x^{[c]} = \{1\}$ for every vertex $x \in \prod V X_i$. 

\item \label{it:GraphsC} 
Assume that $X_i$ is coarsely simply connected for all $i$ (e.g. if $X_i$ is a tree), and let $r>0$. Then there is, up to the action of $\Aut(G)$, only finitely many cocompact lattices $\Gamma \leq G$ with at most $r$ orbits on $\prod VX_i$ and such that $p_i(\Gamma)$ is dense in $G_i$ for all $i$. 
\end{enumerate}
\end{corintro}

Statement \ref{it:GraphsA} extends the main result of \cite{BuMo_Wang} from quasi-primitive local actions of almost simple type, to arbitrary semiprimitive local actions. As before, Serre's covolume formula immediately implies that, for a fixed $r >0$, the set of covolumes of lattices as in \ref{it:GraphsB} is a finite set. Hence Statement \ref{it:GraphsB} provides a partial answer to Question~1.2 from \cite{BuMo_Wang}. 

\bigskip

Our tools can also be used to study discrete subgroups $\Gamma \leq \Aut(X_1) \times \dots \times \Aut(X_n)$ with semiprimitive local action on each factor, without assuming that the closure of the projection $p_i(\Gamma) \leq \Aut(X_i)$ is a fixed closed subgroup $G_i$. In particular we prove the following: 

\begin{thmintro} \label{thm-intro-prod-trees-2trans}
Let $n \geq 1$ and for each $i = 1, \ldots, n$, let $T_i$ be a regular  locally finite tree of degree~$\geq 3$. 
	There is, up to conjugation, only finitely many vertex-transitive discrete subgroups $\Gamma \leq \Aut(T_1) \times \dots \times \Aut(T_n)$ whose local action on $T_i$ is $2$-transitive for all $i$. 
\end{thmintro}

Note that, while the hypotheses of  Corollary~\ref{corintro:graphs} imply that the number $n$ of factors is at least~$2$, the case $n=1$ is allowed and meaningful in Theorem~\ref{thm-intro-prod-trees-2trans}: in that case,  the result is due to Trofimov--Weiss~\cite[Theorem~1.4]{TrofimovWeiss}, and relies on the Classification of the Finite Simple Groups. Our proof of Theorem~\ref{thm-intro-prod-trees-2trans} consists of a reduction from the case $n>1$ to the case of a single factor.

Recall that a conjecture of R.~Weiss~\cite[Conjecture~3.12]{Weiss78}  predicts that the conclusion of Theorem~\ref{thm-intro-prod-trees-2trans} with $n=1$ holds for discrete groups with primitive local actions. Although it is still open in full generality, several cases of the conjecture are now known to be true (see \S \ref{subsec-prod-graph-semiprim} for details). Our results in Section~\ref{sec:graphs} provide a reduction from the case of $n$ factors to the case of a single factor; see Theorem \ref{thm-prod-trees-IL-finite}. In particular we provide a partial solution to a conjecture due to Y.~Glasner \cite[Conjecture~1.5]{Glasner-2dim}. 

\medskip 
In the rest of the introduction, we discuss some of the proof ingredients of the results presented above. 

\subsection{A substitute for a theorem of Zassenhaus}

For a locally compact group $G$, we write $\sub(G)$ for the space of closed subgroups of $G$. Endowed with the Chabauty topology \cite{Chabauty-topo}, the space $\sub(G)$ is compact.

It is a classical fact that a connected Lie group can be approximated by its discrete subgroups in the Chabauty topology only if it is nilpotent \cite{Kuran}. In particular a non-compact simple Lie group $G$ cannot be approximated by discrete subgroups. Actually in this setting $G$ is an isolated point in $\sub(G)$ \cite[Proposition 13.6]{Gel-irs}. This is no longer true for compactly generated simple tdlc groups. For example the group $\aut(T)^+$ of automorphisms of a regular tree $T$ can be approximated by proper subgroups. An explicit sequence of subgroups approximating $\aut(T)^+$ may be found within the family of groups constructed by N.~Radu in \cite{Radu-groups} (see the appendix in \cite{Cap-Rad-chab}).

The proofs of Theorems~\ref{thm:Wang1},~\ref{thmintro:KazMar} and~\ref{thmintro:Wang2} are based on a study of the collection $\mathcal L_K$ of all $K$-cocompact discrete subgroups $\Gamma \leq G = G_1 \times \dots \times G_n$ with dense projections in a product of quasi just-non-compact groups. A key point that we establish is the fact that $\mathcal L_K$ is a closed subset of $\sub(G)$, and is therefore compact (see Theorem~\ref{thm:LrU-ChabClosed} below). An important ingredient of independent interest in the proof of Theorem~\ref{thm:LrU-ChabClosed} is the following result, which provides sufficient conditions for a non-discrete, compactly generated, quasi just-non-compact tdlc group to be isolated from the set of its discrete subgroups in the Chabauty topology.

\begin{thmintro}\label{thmintro:Zassenhaus}
Let $G$ be a non-discrete, compactly generated, quasi just-non-compact tdlc group. If an open subgroup of finite index in $G$ is a Chabauty limit of a sequence of discrete subgroups of $G$, then there is a prime $p$ and a compact open subgroup $V \leq G$ such that $V$ is a pro-$p$ group that is  not topologically finitely generated. 
\end{thmintro}

That result may be interpreted as an analogue of the classical result of Zassenhaus ensuring that every Lie group $G$ has an identity neighbourhood $U$ such that for each discrete subgroup $\Gamma \leq G$, the intersection $\Gamma \cap U$ is contained in a connected nilpotent Lie subgroup of $G$ (see \cite{Zassenhaus} and \cite[Theorem~8.16]{Raghu}). The proof of Theorem~\ref{thmintro:Zassenhaus} relies in an essential way on results from \cite{CRW-part2} and \cite{CRW_dense} on  the structure of   tdlc groups and their locally normal subgroups. It is also inspired by the proof of the Thompson--Wielandt theorem in finite group theory and its variants for discrete automorphism groups of graphs with primitive local action (see e.g.\ \cite[Theorem~2.1.1]{BuMo1}). 

\subsection{Chabauty neighbourhoods of lattices in Kazhdan groups}

The proof of Theorem~\ref{thm:Wang1} in the case of non-uniform lattices in Kazhdan groups also relies on Chabauty considerations, and essentially splits into two parts. The first one is similar to the aforementioned Theorem~\ref{thm:LrU-ChabClosed}, and consists in proving that, in the setting of Theorem~\ref{thm:Wang1}, the collection of discrete subgroups of $G$ containing a non-uniform lattice $\Gamma$ is Chabauty closed (see Proposition \ref{prop-kazhdan-no-tower}). The second one is given by the following additional result which is of independent interest. 

\begin{thmintro}\label{thm:Nbd}
	Let $G$ be a locally compact group with Kazhdan's property (T), such that the amenable radical $R(G)$ is discrete. Then the set of lattices in $G$ forms an open subset of the Chabauty space $\sub(G)$. More precisely, for any lattice $\Gamma \leq G$, there is an identity neighbourhood $U$ in $G$ such that the set of those  lattices $\Lambda \leq G$ with $\Lambda \cap U = \{1\}$ is a neighbourhood of $\Gamma$ in $\sub(G)$. 
\end{thmintro} 

The proof uses an important continuity property of induction of unitary representations, due to J.~Fell~\cite{Fell1964}, which implies that the set of closed subgroups of finite covolume in a second countable locally compact group with property (T) is Chabauty open (see Theorem~\ref{thm:Fell}).

\subsection{Irreducibility}

Another important point in the proof of  Theorems~\ref{thm:Wang1},~\ref{thmintro:KazMar} and~\ref{thmintro:Wang2} is the notion of \emph{irreducibility} for lattices in products. While the density of the projections of a lattice in a product of $2$ non-discrete factors can be viewed as a condition of irreducibility, this is no longer the case for a product of $n> 2$ factors. Indeed, a lattice $\Gamma$ in a product of~$4$ factors $G_1 \times \dots \times G_4$ can have dense projections and  be the direct product of two subgroups, that are respectively lattices in $G_1 \times G_2$ and $G_3 \times G_4$. Since that issue is directly relevant to the proofs of the results above in case of more than $2$ factors, we take this opportunity to identify various possible definitions of irreducibility for lattices in products, and clarify the logical relations between them.  

Consider again a  product group $G = G_1 \times \ldots \times G_n$.   For $\Sigma \subseteq \left\{1,\ldots,n\right\}$, we   denote the associated sub-product by $G_\Sigma = \prod_{j \in \Sigma} G_j$. We   denote by $p_\Sigma \colon G \rightarrow G_\Sigma$ the projection on $G_\Sigma$. We identify $G_\Sigma$ with its natural image in $G$. When $\Sigma = \{i\}$ is a singleton, the projection to $G_i$ is denoted by $p_i$, as above. 

Assume now that  $G_1,\ldots,G_n$ are non-discrete locally compact groups  and let  $\Gamma \leq G = G_1 \times \ldots \times G_n$ be a   lattice.  
Each of the following conditions,  may be seen as expressing the fact that $\Gamma$ is \textit{irreducible}. 

\begin{enumerate}[label=\textbf{(Irr\arabic*)}]
\setcounter{enumi}{-1}
	\item  \label{it:irre0} For every partition $ \Pi \cup \Sigma = \left\{1,\ldots,n\right\}$ with $\Pi, \Sigma \neq \varnothing$, the subgroup  $(G_\Pi \cap \Gamma)(G_\Sigma \cap \Gamma)$ is of infinite index in $\Gamma$. 
	
	\item \label{it:irre1}  For every $\Sigma \subsetneq \left\{1,\ldots,n\right\}$, the projection $p_\Sigma \colon \Gamma \to G_\Sigma$ has  dense image. 	
\end{enumerate} 
\begin{enumerate}[label={\textbf{(Irr\arabic*$'$)}}]
	\item \label{it:irre1'}  For every $\Sigma \subsetneq \left\{1,\ldots,n\right\}$, the closure of the image of $p_\Sigma \colon \Gamma \to G_\Sigma$ contains $\prod_{i \in \Sigma} G^+_i$, where $G^+_i$ is a non-discrete closed cocompact normal subgroup of $G_i$.  
\end{enumerate}		
	\begin{enumerate}[label=\textbf{(Irr\arabic*)}]
		\setcounter{enumi}{1}
	\item \label{it:irre2} For every non-empty $\Sigma \subseteq \left\{1,\ldots,n\right\}$, the projection $p_\Sigma \colon \Gamma \to G_\Sigma$ is injective. Equivalently, $p_i \colon \Gamma \to G_i$ is injective for every $i$.
	
	\item \label{it:irre3}  For every non-empty $\Sigma \subsetneq \left\{1,\ldots,n\right\}$, the projection $p_\Sigma \colon \Gamma \to G_\Sigma$ has a non-discrete image.

	\item \label{it:irre4} If a subgroup $\Lambda \leq \Gamma$ is isomorphic to a direct product of two non-trivial groups, then the index $|\Gamma : \Lambda|$ is infinite. 
	
\end{enumerate} 

As mentioned above, when $n=2$, it is customary to adopt the condition \ref{it:irre1} as the definition of the irreducbility of $\Gamma$. 
In his book \cite{Margulis}, Margulis studies the case where each factor is an absolutely almost simple algebraic group $\mathbf G_i$ over a non-discrete locally compact field $k_i$. In that context, he adopts the condition \ref{it:irre0} as the definition of the  \textit{irreducibility} of $\Gamma$ (see \cite[Definition II.6.5]{Margulis}). Using    \cite[Theorem II.6.7]{Margulis}  and \cite[Theorem IV.4.10]{Margulis}, if follows that if $\mathbf G_i$ is adjoint and of $k_i$-rank~$\geq 1$,  then the  irreducibility conditions \ref{it:irre0},  \ref{it:irre1'}, \ref{it:irre2},  \ref{it:irre3} and  \ref{it:irre4}   are all equivalent.  The conditions \ref{it:irre1}--\ref{it:irre4} are considered in \cite[\S2.B]{CaMoKM}. Under the hypothesis that each $G_i$ is an isometry group of a proper CAT($0$) space satisfying suitable natural conditions, it is shown in  \cite[Proposition~2.2]{CaMoKM} that \ref{it:irre2}~$\Rightarrow$~\ref{it:irre3}~$\Rightarrow$~\ref{it:irre4}, while the implication \ref{it:irre4}~$\Rightarrow$~\ref{it:irre2} generally fails in that context.

The following result shows how those conditions relate to one another in case of compactly generated quasi just-non-compact groups.

\begin{thmintro}\label{thm:irred:JNC}
	Let $G_1,\ldots,G_n$ be non-discrete, compactly generated, tdlc groups and  $\Gamma \leq G = G_1 \times \ldots \times G_n$ be a cocompact lattice such that $p_i \colon \Gamma \rightarrow G_i$ has dense image for every $i=1,\ldots,n$.  
	\begin{enumerate}[label=(\roman*)]
		\item \label{item-thmirred-qJNC} If $G_i$ is quasi just-non-compact for all $i$, then
		\begin{center}
			{\rm  \ref{it:irre2} $\Rightarrow$ \ref{it:irre0} $\Leftrightarrow$ \ref{it:irre1'} $\Leftrightarrow$ \ref{it:irre3} $\Leftarrow$ \ref{it:irre4}.}
		\end{center}
		\item \label{item-thmirred-JNC} If $G_i$ is just-non-compact for all $i$, then
		\begin{center}
			{\rm  \ref{it:irre0} $\Leftrightarrow$ \ref{it:irre1'} $\Leftrightarrow$ \ref{it:irre2} $\Leftrightarrow$ \ref{it:irre3} $\Leftarrow$ \ref{it:irre4}.}
		\end{center}
		
		\item \label{item-thmirred-hJNC} If every finite index open subgroup of $G_i$ is just-non-compact for all $i$, then
		\begin{center}
			{\rm  \ref{it:irre0} $\Leftrightarrow$ \ref{it:irre1'} $\Leftrightarrow$ \ref{it:irre2} $\Leftrightarrow$ \ref{it:irre3} $\Leftrightarrow$ \ref{it:irre4}.}
		\end{center}
	\end{enumerate}
\end{thmintro}

We make several comments about the statement:

\begin{enumerate}[label=\arabic*)]

\item Every non-discrete, compactly generated, quasi just-non-compact tdlc group $G$ has a smallest cocompact normal subgroup, denoted by $G^{(\infty)}$; it is non-discrete and coincides with the intersection of all finite index open subgroups of $G$ (see Proposition~\ref{prop-carac-qjnc}). Thus $G_i^+$ may be replaced by $G_i^{(\infty)}$ in \ref{it:irre1'} when $G_i$ is quasi just-non-compact.

\item Among the above implications, \ref{it:irre3} $\Rightarrow$ \ref{it:irre1'} is the most significant one. Together with \ref{it:irre1'} $\Rightarrow$ \ref{it:irre2}, they are the only ones where the assumption that the factors are (quasi) just-non-compact is crucially used. We refer to Section~\ref{sec:Proofs} for more precise information on the logical relations amongst the conditions \ref{it:irre0}--\ref{it:irre4} under weaker assumptions on the factors.

\item The implication \ref{it:irre0} $\Rightarrow$ \ref{it:irre2} in statement \ref{item-thmirred-qJNC} does not hold. Explicit examples show that a cocompact lattice $\Gamma$ with dense projections in a product of quasi just-non-compact {}groups may fail to satisfy \ref{it:irre2} (see from \cite{BuMo2} and \cite{CaMoKM}). Such lattices are usually not residually finite (see \cite[Proposition~2.5]{CaMoKM}, \cite[\S5.4]{CapWes} and Corollary~\ref{cor:irred:RamenTrivial} below).  See also Corollary~\ref{cor:irred:RamenTrivial:2} for conditions that are equivalent to \ref{it:irre2} in the context of Theorem~\ref{thm:irred:JNC} \ref{item-thmirred-qJNC}.

\item Notice that  \ref{it:irre4} can  be viewed as a condition involving all finite index subgroups of $\Gamma$. It is thus not surprising that its equivalence with the other irreducibility conditions requires a hypothesis on $G$ that is robust enough to be inherited by the finite index open subgroups of $G$.

\item The hypothesis of cocompactness of $\Gamma$ in Theorem~\ref{thm:irred:JNC} and its corollaries  is only used to ensure that  the intersection of $\Gamma$ with suitable open subgroups of $G$ are finitely generated. Those statements remain true for non-uniform lattices if one assumes in addition that $G$ has Kazhdan's property (T); see Propositions~\ref{prop-(T)-0} and \ref{prop-H=G} and Remark \ref{rmq-(T)-2} for appropriate modifications in the proofs.

\item The Bader--Shalom Normal Subgroup Theorem \cite{BaSha} deals with irreducible lattices in products of just-non-compact groups. Combining Theorem~\ref{thm:irred:JNC} with their result, we shall establish  a version of the Normal Subgroup Theorem  for lattices in products of quasi-just-non-compact groups,  see Corollary~\ref{cor:irred:RamenTrivial:2} below. 
\end{enumerate}

A simple decomposition process allows, given a lattice $\Gamma$ with dense projections in a product of quasi just-non-compact groups, to decompose $\Gamma$ (up to passing to a finite index overgroup) as a product of factors all satisfying the irreducibility condition \ref{it:irre3} (see Corollary \ref{cor:ReductionToIrr}). Theorem~\ref{thm:irred:JNC} then ensures that each piece also satisfies \ref{it:irre0} and \ref{it:irre1'}. As a by-product, we obtain the following supplementary result for lattices in the product of \textit{three} quasi just-non-compact tdlc groups.

\begin{corintro} \label{cor:3factors:JNC}
	Let $G_1,G_2,G_3$ be non-discrete, compactly generated,  quasi just-non-compact tdlc groups. Let $\Gamma \leq G_1 \times G_2 \times G_3$ be a cocompact lattice such that $p_i \colon \Gamma \rightarrow G_i$ has dense image for $i=1,2,3$. Then the equivalent conditions \ref{it:irre0}, \ref{it:irre1'} and \ref{it:irre3} from Theorem~\ref{thm:irred:JNC}(i) are automatically satisfied.
\end{corintro}

\subsection*{Organization}

Section \ref{sec-structure-tdlc} contains preliminary results about general compactly generated tdlc groups, with a special emphasis on quasi just-non-compact groups in \S \ref{subsec-QJNC}.

In Section \ref{sec-approx-QJNC} we study Chabauty approximations of non-discrete, compactly generated, quasi just-non-compact tdlc groups, and prove Theorem \ref{thmintro:Zassenhaus}. This section, notably \S \ref{subsec-(r,U)coc}, also contains preliminary results that are used in later sections.

Section \ref{sec-irr} deals with the aforementioned irreducibility conditions for lattices in products, and contains the proofs of Theorem \ref{thm:irred:JNC} and Corollary \ref{cor:3factors:JNC}. An important ingredient for these proofs is Proposition \ref{prop-H=G}.

The main task of Section \ref{sec:Wang} is to prove Theorem \ref{thm:LrU-ChabClosed}, which is the key intermediate step in  the proofs of Theorems~\ref{thm:Wang1},~\ref{thmintro:KazMar} and~\ref{thmintro:Wang2}. The proof of Theorem \ref{thm:LrU-ChabClosed} relies on all the previous sections, and notably on Proposition \ref{prop-H=G}. 

Section \ref{sec:graphs} is concerned with discrete groups acting on product of graphs with semiprimitive local action on each factor, and contains the proof of Corollary~\ref{corintro:graphs}. We also prove additionnal results, namely Theorem \ref{thm-prod-trees-IL-finite} and Corollary \ref{cor:TreesFinite}, from which we deduce Theorem~\ref{thm-intro-prod-trees-2trans}.

Finally the proof of Theorem~\ref{thm:Nbd}, which is used in Section \ref{sec:Wang} in the proof of Theorem~\ref{thm:Wang1} in the case of Kazhdan groups, is presented in an appendix.

\subsection*{Acknowledgements}

We are grateful to Richard Weiss for stimulating discussions and for his interest in this work. We also thank the referee for a careful reading of the article and for useful comments and suggestions.

\section{On the structure of tdlc groups} \label{sec-structure-tdlc}

\subsection{Preliminaries}
A locally compact group $G$ is \textbf{residually discrete} if the intersection of all open normal subgroups of $G$ is trivial. We will invoke the following result from \cite[Corollary~4.1]{CaMo-decomp}.

\begin{prop} \label{prop:CapMon-RD}
A compactly generated locally compact group $G$ is residually discrete if and only if the compact open normal subgroups form a basis of identity neighbourhoods in $G$. 
\end{prop}

The following result is a direct consequence of a very general result due to Colin Reid \cite[Theorem~G]{ReidDistal}. We provide a direct proof for the reader's convenience. 

\begin{prop} \label{prop-exten-RD}
Let $G$ be a compactly generated locally compact group with a discrete normal subgroup $N$ such that $G/N$ has a basis of identity neighbourhoods consisting of compact open normal subgroups. Then $G$ also has this property. 
\end{prop}

\begin{proof}
If this is not the case, then we may find a compact open subgroup $U$ such that $U \cap N = 1$ and $U$ contains no open normal subgroup of $G$. Note that this implies in particular that $G$ is not discrete.

Let $K$ be the normal core of $U$ in $G$, and $\pi: G \to G/K$ the quotient map. Note that $\pi(U)$ is a compact open subgroup of $\pi(G)$ that contains no non-trivial normal subgroup of $\pi(G)$. Moreover the group $\pi(G)$ admits $\pi(N)$ as a discrete normal subgroup, and the quotient $\pi(G) / \pi(N) \simeq G / KN$ has a basis of identity neighbourhoods consisting of compact open normal subgroups since it is a quotient of $G/N$. So in order to obtain our contradiction we may assume that $K=1$.

Consider the family $\F$ of open normal subgroups of $G$ containing $N$. By Proposition~2.5 in \cite{CaMo-decomp}, the group  $\bigcap_\F M$ intersects $U$ non-trivially. But the assumption on $G/N$ implies that the subgroup $\bigcap_\F M$ is equal to $N$. So we obtain that $U \cap N$ is non-trivial, which is our contradiction. 
\end{proof}

A locally compact group is \textbf{locally elliptic} if every compact subset is contained in a compact subgroup, and the \textbf{locally elliptic radical} is the largest normal subgroup that is locally elliptic. The \textbf{polycompact radical} $W(G)$ of a locally compact group $G$ is the union of all compact normal subgroups of $G$. As observed by Tits, the subgroup $W(G)$ is not necessarily closed in $G$. The examples given in \cite[Proposition 3]{Tits-depl-born} are such that every compact normal subgroup is finite, but the subgroup that they generate is a proper dense subgroup. However when $G$ is compactly generated, $W(G)$ is a closed subgroup of $G$: see \cite[Theorem 1.2]{Cor-comm-focal} and references given there, notably \cite{TrofimovBG}. In particular for a compactly generated group, the property that $W(G)$ is discrete, considered repeatedly in the article, is equivalent to the property that every compact normal subgroup of $G$ is finite.

The following  fact will be used repeatedly. 

\begin{prop}\label{prop:CocoNormal}
Let $G$ be a locally compact group and $H$ a cocompact closed subgroup. Then every compact normal subgroup of $H$ is contained in a compact normal subgroup of $G$. 
\end{prop}

\begin{proof}
By \cite[Proposition 2.7]{Cor-comm-focal}, every compact normal subgroup $K$ of $H$ is contained in the polycompact radical $W(G)$. By definition, the polycompact radical is a subgroup of the locally elliptic radical (see \cite[Proposition~2.4(6)]{Cor-comm-focal}). Since the normalizer of $K$ in $G$ contains $H$, and is thus cocompact in $G$, and it follows that the union $X$ of the conjugacy classes of elements of $K$ is a compact subset of $G$. Since $X$ is contained in the locally elliptic radical of $G$, it follows that the closed subgroup generated by $X$ is compact. This subgroup is also normal in $G$ and contains $K$, hence the statement.
\end{proof}

Following Burger--Mozes, the \textbf{quasi-center} of a locally compact group $G$ is the set $\QZ(G)$ of those elements whose centralizer is open. The quasi-center is a (possibly non-closed) topologically characteristic subgroup of $G$ containing all discrete normal subgroups. Recall that a subgroup is \textbf{topologically characteristic} if it is invariant by all topological group automorphisms.

\begin{prop} \label{prop:NormalizerLatticeBis}
Let $G$ be a compactly generated tdlc group, and $\Gamma$ a cocompact lattice in $G$. Then there exists a compact normal subgroup $K \triangleleft G$ such that $K \Gamma \cap N_G(\Gamma)$ has finite index in $N_G(\Gamma)$.
\end{prop}

\begin{proof}
The subgroup $\Gamma$ is cocompact in the compactly generated group $G$, so $\Gamma$ is finitely generated. Since $\Gamma$ is a discrete normal subgroup of $N_G(\Gamma)$, $\Gamma$ lies in the quasi-center of $N_G(\Gamma)$, i.e.\ every element of $\Gamma$ centralizes an open subgroup of $N_G(\Gamma)$. But since $\Gamma$ is finitely generated, it follows that the centralizer $C_G(\Gamma)$ is open in $N_G(\Gamma)$. If $U$ is a compact open subgroup of $N_G(\Gamma)$ centralized by $\Gamma$, then $U \Gamma$ is a subgroup that is open and cocompact in $N_G(\Gamma)$ since $\Gamma$ is cocompact in $G$. Therefore $U \Gamma$ has finite index in $N_G(\Gamma)$. Now the compact subgroup $U$ has a cocompact normalizer in $G$, so it follows from Proposition \ref{prop:CocoNormal} that $U$ is contained in a compact normal subgroup $K$ of $G$, and $K \Gamma$ indeed contains a finite index subgroup of $N_G(\Gamma)$.
\end{proof}

Proposition \ref{prop:NormalizerLatticeBis} has the following consequence:

\begin{cor} \label{cor:NormalizerLattice}
Let $G$ be a compactly generated tdlc group with a discrete polycompact radical. Then every cocompact lattice $\Gamma \leq G$   has finite index in its normalizer $N_G(\Gamma)$.
\end{cor}

\begin{proof}
Let $K$ be as in the conclusion of Proposition \ref{prop:NormalizerLatticeBis}. Since $G$ has a discrete polycompact radical, the subgroup $K$ must be finite, and it follows that $\Gamma$ has finite index in $K \Gamma \cap N_G(\Gamma)$. Since $K \Gamma \cap N_G(\Gamma)$ has finite index in $N_G(\Gamma)$, the conclusion follows.
\end{proof}

The special case of the following result when $G$ is topologically simple is due to Barnea--Ershov--Weigel \cite[Theorem~4.8]{BEW}. Recall that a group $G$ is \textbf{topologically characteristically simple} if the only closed topologically characteristic subgroups of $G$ are the trivial ones.

\begin{prop}
	\label{prop:trivialQZ}
Let $G$ be a  non-discrete, non-compact,    compactly generated tdlc group that is topologically characteristically simple. Then $\QZ(G)= \{1\} = W(G)$.  
\end{prop}

\begin{proof}
	The polycompact radical $W(G)$ is a topologically characteristic subgroup. Assume it is non-trivial. Thus it is dense by hypothesis. Since $G$ is compactly generated, the radical $W(G)$ is closed by  \cite[Theorem 1.2]{Cor-comm-focal}, so $G = W(G)$. Since the polycompact radical is a subgroup of the locally elliptic radical (see \cite[Proposition~2.4(6)]{Cor-comm-focal}), it follows that $G$ is locally elliptic. Since $G$ is compactly generated, this implies that $G$ is compact, contradicting the hypotheses. Hence $W(G)=\{1\}$. 
	
	We next observe that the quasi-center is also a topologically characteristic subgroup. Assume that $\QZ(G)$ is non-trivial. Thus it is dense by hypothesis. It then follows from \cite[Proposition~4.3]{CaMo-decomp} that the compact open normal subgroups of $G$ form a basis of identity neighbourhoods. In particular $W(G)$ is dense, contradicting the first part of the proof. 	
\end{proof}

\subsection{Monolithic and just-non-compact groups}

A locally compact group $G$ is  \textbf{monolithic} if the intersection of all non-trivial closed normal subgroups is itself non-trivial. That intersection is then called the \textbf{monolith} of $G$, denoted by $\Mon(G)$. Clearly, every non-trivial minimal closed normal subgroup in a locally compact group must be topologically characteristically simple. Thus a source of groups to which Proposition~\ref{prop:trivialQZ} applies is given by the monolith of a compactly generated tdlc group, provided that this monolith is cocompact. 

The following consequence of \cite[Theorem~E]{CaMo-decomp} clarifies the link between monolithic groups with a cocompact monolith and just-non-compact groups. A locally compact group is called \textbf{hereditarily just-non-compact} if every finite index open subgroup is just-non-compact.  

\begin{prop}\label{prop:JustNonCompact} 
	Let $G$ be a compactly generated non-discrete tdlc group. 
	\begin{enumerate}[label=(\roman*)]
		\item $G$ is just-non-compact  if and only if $G$ is monolithic with a non-discrete cocompact monolith. In particular if $G$ is just-non-compact, then $\QZ(G) = 1$.
		
		\item $G$ is hereditarily just-non-compact    if and only if $G$ is monolithic with a non-discrete cocompact topologically simple monolith. 
	\end{enumerate}
\end{prop}

\begin{proof}
	For both (i) and (ii), the `only if' implications follow from  \cite[Theorem~E]{CaMo-decomp}. Conversely, if $G$ is compactly generated and monolithic such that $\Mon(G)$  is cocompact, then by the definition of monolithicity, every non-trivial closed normal subgroup of $G$ contains $\Mon(G)$, and is thus cocompact in $G$. Moreover $G$ cannot be compact, since a compact tdlc group is profinite, hence residually finite, and can thus not be monolithic. Thus the condition of part (i) is indeed sufficient. The fact that $G$ has trivial quasi-center then follows from Proposition~\ref{prop:trivialQZ}.
	
	Assume now that $G$ is compactly generated and monolithic such that $\Mon(G)$  is cocompact and topologically simple. Let $H$ be an open subgroup of finite index in $G$. Then $H$ contains an open normal subgroup of finite index in $G$; in particular $H$ contains $\Mon(G)$. Given a non-trivial closed normal subgroup $N$ of $H$, then $N \cap \Mon(G)$ is normal in $\Mon(G)$. Since the latter is topologically simple by hypothesis, we have $N \geq \Mon(G)$ or $N \cap \Mon(G) = 1$. In the latter case we deduce that $N \leq \mathrm C_G(\Mon(G))$. That centralizer is thus a non-trivial closed normal subgroup of $G$. It must contain $\Mon(G)$ by definition. This implies that $\Mon(G)$ is self-centralizing, hence abelian. This contradicts the hypothesis that $\Mon(G)$ is topologically simple and non-discrete. We conclude that $N \geq \Mon(G)$. Thus every non-trivial closed normal subgroup of $H$ contains $\Mon(G)$ and is thus cocompact. This proves that $G$ is hereditarily just-non-compact, as required. 	
\end{proof}

\subsection{Quasi just-non-compact groups} \label{subsec-QJNC}

\begin{defin}
	Let $G$ be a locally compact group. We say that $G$ is \textbf{quasi just-non-compact} if $G$ is non-compact, and every closed normal subgroup of $G$ is either discrete or cocompact. %It is \textbf{(topologically) quasi-simple} if every proper closed normal subgroup is discrete. 
\end{defin}

Following \cite{BuMo1}, for a tdlc group $G$, we denote by $G^{(\infty)}$ the intersection of all finite index open subgroups of $G$. The subgroup $G^{(\infty)}$ is closed and topologically characteristic in $G$, and coincides with the intersection of all closed cocompact normal subgroups of $G$.

\begin{prop} \label{prop-carac-qjnc}
	Let $G$ be a non-discrete compactly generated tdlc group. The following are equivalent:
	\begin{enumerate}[label=(\roman*)]
		\item \label{it-qjnc} $G$ is quasi just-non-compact;
		
		\item \label{it-qjnc-caract} $\QZ(G)$ is a discrete non-cocompact subgroup of $G$; $G^{(\infty)}$ is a non-discrete cocompact subgroup of $G$, and every closed normal subgroup of $G$ is either contained in $\QZ(G)$ or contains $G^{(\infty)}$.  
	\end{enumerate}
If those conditions hold, then  $G/\QZ(G)$ is just-non-compact.
\end{prop}

\begin{proof}
	That \ref{it-qjnc-caract} implies \ref{it-qjnc} is clear by observing that non-discreteness of $G^{(\infty)}$ implies that $G$ cannot be compact. We show that \ref{it-qjnc} implies \ref{it-qjnc-caract}. Assume for contradiction that $\QZ(G)$ is not discrete. Then by the assumption that $G$ is quasi just-non-compact, the subgroup $H = \overline{\QZ(G)}$ is cocompact in $G$. According to \cite[Proposition 4.3]{CaMo-decomp}, $H$ admits a compact open normal subgroup $U$, which by Proposition \ref{prop:CocoNormal} must be contained in a compact normal subgroup of $G$. Since every compact normal subgroup of $G$ is finite, it follows that $U$ is finite, and $H$ is discrete. So $\QZ(G)$ was actually discrete. 
	
	Observe that since $G$ is quasi just-non-compact and non-discrete, an open normal subgroup of $G$ necessarily has finite index in $G$. It follows that $G^{(\infty)}$ coincides with the intersection of all open normal subgroups of $G$, i.e.\ $G / G^{(\infty)}$ is a residually discrete group. Therefore by Proposition~\ref{prop:CapMon-RD} the group $G / G^{(\infty)}$ admits a basis of identity neighbourhoods consisting of compact open normal subgroups. If $G^{(\infty)}$ is discrete, then it follows from Proposition \ref{prop-exten-RD} that $G$ admits compact open normal subgroups. These are necessarily finite by \ref{it-qjnc}, and it follows that $G$ is discrete, a contradiction. So $G^{(\infty)}$ cannot be discrete, and hence $G^{(\infty)}$ is cocompact in $G$. To obtain \ref{it-qjnc-caract} it only remains to observe that $\QZ(G)$ cannot be cocompact, since othherwise it would contain $G^{(\infty)}$, which would therefore be discrete, a contradiction. 
	
Finally if $N$ is a closed normal subgroup of $G$ containing $\QZ(G)$ properly, then $N$ contains $G^{(\infty)}$, so $N$ is cocompact in $G$. Thus $N/\QZ(G)$ is cocompact in $G/\QZ(G)$, thereby confirming that $G/\QZ(G)$ is just-non-compact.
\end{proof}

More information on the structure of compactly generated just-non-compact groups (hence of quasi just-non-compact groups) may be found in \cite{CaMo-decomp}, \cite{CRW-part2} and \cite{CRW_dense}.  Following those references, we say that a subgroup of a tdlc group is \textbf{locally normal} if its normalizer is open. For the sake of future references, we record the following results.

\begin{prop}[{\cite[Corollary~8.2.4]{CRW_dense}}]\label{prop:RobustlyMonolithic}
Let $G$ be a non-discrete, compactly generated tdlc group. Assume that $G$ is monolithic, with non-discrete, topologically simple, cocompact monolith. For any prime $p$, if $G$ has a non-trivial compact locally normal subgroup that is pro-$p$, then $G$ has a compact open subgroup that is pro-$p$. 
\end{prop}

\begin{prop}\label{prop:JNC:primeContent}
Let $G$ be a non-discrete compactly generated just-non-compact tdlc group. Then for any prime $p$, if $G$ has a non-trivial compact locally normal subgroup that is pro-$p$, then $G$ has a compact open subgroup that is pro-$p$. 
\end{prop}

\begin{proof}
By Proposition~\ref{prop:JustNonCompact}, the group $G$ is monolithic with non-discrete cocompact monolith $M$, and $\QZ(G)=1$. Moreover $\QZ(M) = 1$ by Proposition~\ref{prop:trivialQZ}. In particular $M$ is not abelian, so that $\mathrm C_G(M)= \{1\}$ since otherwise we would have $M \leq \mathrm C_G(M)$ by definition of the monolith.

By \cite[Theorem~E]{CaMo-decomp}, the monolith $M$ has finitely many minimal closed normal subgroups $N_1, \dots, N_\ell$ that are topologically simple, and $M = \overline{N_1\dots N_\ell}$. In particular the $G$-action by conjugation on $M$ permutes that $N_i$ transitively. Let $G_0 \trianglelefteq G$ the open normal subgroup of finite index that normalizes $N_i$ for all $i$. For every finite index open subgroup $G_1 \leq G_0$, we have $M \leq G_1$ since $N_i \cap G_1$ is open of finite index in $N_i$, and $N_i$ is topologically simple. Moreover every non-trivial closed normal subgroup $N$ of $G_1$ contains one of the $N_i$, since otherwise $N$ would commute with $N_i$ for all $i$, hence be contained in $\mathrm C_G(M) = \{1\}$. Thus $N_1, \dots, N_\ell$ are the minimal closed normal subgroups of $G_1$. It then follows from \cite[Corollary~3.3]{CRW-part2} that $G_1/\mathrm C_{G_1}(N_i)$ is monolithic with monolith $\overline{N_i \mathrm C_{G_1}(N_i)}/ \mathrm C_{G_1}(N_i)$. Notice that this monolith contains $\overline{M\mathrm C_{G_1}(N_i)}/\mathrm C_{G_1}(N_i)$ and is thus cocompact. This shows that $ R_i = G_0/\mathrm C_{G_0}(N_i)$ is hereditarily just-non-compact. It then follows from Proposition~\ref{prop:JustNonCompact}(ii) that the monolith of $R_i$ is non-discrete, topologically simple and cocompact. Therefore,  if $R_i$ has a non-trivial compact locally normal subgroup that is pro-$p$, then $R_i$ is pro-$p$ by Proposition~\ref{prop:RobustlyMonolithic}. 
 
 Notice that $\bigcap_{i = 1}^\ell \mathrm C_{G_0}(N_i) = \mathrm C_{G_0}(M) \leq \mathrm C_G(M) =\{1\}$. Thus the product homomorphism $G_0 \to \prod_{i=1}^\ell R_i$ is injective. Let now $L$ be a non-trivial compact locally normal subgroup  of $G$ that is pro-$p$. If $L \cap G_0 =\{1\}$, then $L$ is finite, hence contained in $\QZ(G)=\{1\}$, a contradiction. Thus we may assume without loss of generality that $L\leq  G_0$. Since the projection $G_0 \to \prod_{i=1}^\ell R_i$ is injective, there is $i$ such that the projection of $L$ to $R_i$ has a non-trivial image. It then follows from the previous paragraph that $R_i$ is locally pro-$p$. Since $G$ normalizes $G_0$ and acts transitively on the set $\{N_1, \dots, N_\ell\}$, the groups $R_1, \dots, R_\ell$ are pairwise isomorphic, so that $R_i$ is locally pro-$p$ for every $i$. In particular, if $V$ is a sufficiently small compact open subgroup of $G$ contained in $G_0$, then the projection $G_0 \to R_i$ maps  $V$ onto a pro-$p$ subgroup of $R_i$. Using again the injectivity of the product map $G_0 \to \prod_{i=1}^\ell R_i$, it follows that $V$ is pro-$p$. This confirms that $G$ has a compact open pro-$p$ subgroup. 
\end{proof}

\section{Chabauty approximations of quasi just-non-compact groups} \label{sec-approx-QJNC}

Recall that we denote by $\sub(G)$ the set of closed subgroups of a locally compact group $G$. The sets of the form $\left\{H \in \sub(G) \, : \, H \cap K = \emptyset; \, H \cap U_i \neq \emptyset \, \, \text{for all $i$} \, \right\}$, where $K \subset G$ is compact and $U_1,\ldots,U_n \subset G$ are open, form a basis for the Chabauty topology on $\sub(G)$. Endowed with this topology, the space $\sub(G)$ is compact. 

This section aims at studying the properties of the closed subgroups of $G$ contained in a sufficiently small neibourhood of $G$ in its Chabauty space $\sub(G)$, under the assumption that $G$ is a non-discrete, compactly generated, quasi just-non-compact tdlc group.

\subsection{On the set of $(r, U)$-cocompact subgroups} \label{subsec-(r,U)coc}

\begin{defin}
	Let $G$ be a tdlc group, $U \leq G$ a compact open subgroup and $r \geq 1$. A closed cocompact subgroup $H \leq G$ is said to be \textbf{$(r,U)$-cocompact} if the double coset space $H \backslash G / U$ has cardinality at most $r$. We will denote by $\Ccal_{r,U}(G)$ the set of $(r,U)$-cocompact subgroups of $G$.
\end{defin}

In the sequel we will make use of the notion of \textbf{Cayley-Abels graph}. Recall that if $G$ is a compactly generated tdlc group and $U$ is a compact open subgroup of $G$, a Cayley-Abels graph of $G$ associated to $U$ is a connected locally finite graph on which $G$ acts vertex-transitively, and with vertex stabilizers the conjugates of $U$. We refer the reader to \cite{CorHar} for a detailed exposition.

\begin{lem} \label{lem-rU-coc-closed}
	Let $G$ be a compactly generated tdlc group, $U \leq G$ a compact open subgroup and $r \geq 1$. Then $\Ccal_{r,U}(G)$ is a clopen subset of $\sub(G)$. 
\end{lem}

\begin{proof}
Let $V$ be the intersection of all conjugates of $U$ in $G$. Since the reduction modulo $V$ induces a well defined continuous surjection $\varphi: \sub(G) \rightarrow \sub(G/V)$ such that $\varphi^{-1}(\Ccal_{r,U/V}(G/V)) = \Ccal_{r,U}(G)$, it is enough to prove the statement when $V$ is trivial. In this case $G$ acts faithfully on any Cayley--Abels graph associated to $U$, and the statement follows from \cite[Proposition 2.6]{Cap-Rad-chab}. 
\end{proof}

\begin{lem} \label{lem-rU-cocom-uniform}
	Let $G$ be a compactly generated tdlc group. Then for all $r,U$ there exists a finite subset $\Sigma \subset G$ such that $H \Sigma U = G$ for every $H$ in $\Ccal_{r,U}(G)$.
	
	Conversely, for every finite subset $\Sigma \subset G$, there is a constant $r$ such that every closed subgroup $H \leq G$ with  $H \Sigma U = G$ belongs to $\Ccal_{r,U}(G)$.
\end{lem}

\begin{proof}
	Let $X$ be a Cayley--Abels graph of $G$ with respect to $U$. Then 	$\Ccal_{r,U}(G)$ consists of the closed subgroups of $G$ acting with at most~$r$ orbits of vertices on $X$. Since $X$ is connected, for every $H \in \Ccal_{r,U}(G)$, there is a set of representatives of the vertex orbits of $H$ that spans a connected subgraph. That subgraph has $r$ vertices, and is thus of diameter~$<r$. 
	
	This shows that the $r$-ball around every vertex of $X$ contains a set of representatives of the $H$-orbits for every $H \in \Ccal_{r,U}(G)$. It now suffices to choose a finite $\Sigma$ such that $\Sigma U$ contains the $r$-ball around $U$ in $X$. 
	
	For the converse assertion, one defines $r$ as the number of left cosets of $U$ in $\Sigma U$, and the result is clear by definition. 
\end{proof}

\begin{lem}\label{lem:rU-intersect-open}
	Let $G$ be a tdlc group and $O \leq G$ be an open subgroup. Then for all $r \geq 1$, all compact open subgroup $U$ and all $(r, U)$-cocompact subgroup $H \leq G$, the intersection $H \cap O$ is $(r, U \cap O)$-cocompact in $O$.
\end{lem}

\begin{proof}
	Indeed, the inclusion $O \leq G$ descends to an injective map of the spaces of double cosets $H \cap O \backslash O /U \cap O \to H\backslash G/U$. 
\end{proof}

\subsection{Approximations of quasi just-non-compact groups by their closed subgroups}

\begin{prop}\label{prop:ChabNeighb-qjnc}
Let $G$ be a non-discrete, compactly generated, quasi just-non-compact tdlc group. Then the collection of compactly generated quasi just-non-compact closed subgroups $H \leq G$ forms a neighbourhood of $G$  in $\sub(G)$.  
\end{prop}	
\begin{proof}
By Proposition~\ref{prop-carac-qjnc}, the quasi-center of $G$ is discrete. Let $U \leq G$ be a compact open subgroup with $U \cap \QZ(G) = \{1\}$. Let $X$ be a Cayley--Abels graph for $G$ with vertex set $G/U$. We claim that $G$ has a Chabauty neighbourhood consisting of cocompact closed subgroups $H$ with the following property: for each closed normal subgroup $N$ of $H$, if $N \cap U \neq \{1\}$ then $N$ is   cocompact in $G$.

If the claim fails, then there is a sequence $(H_k)$ in $\sub(G)$ converging to $G$, and a sequence $(N_k)$ of closed normal subgroups of $H_k$ such that $N_k \cap U \neq \{1\}$ and $N_k$ is not cocompact in $G$ for any $k$. 

First observe that by Lemma~\ref{lem-rU-coc-closed}, there is  a Chabauty neighbourhood of $G$ consisting of   closed subgroups $H$ that are vertex-transitive on $X$ (hence cocompact in $G$). Let $H \leq G$ be closed and vertex-transitive, and let $N$ be a closed normal subgroup of $H$ with $N \cap U \neq 1$. Let $x \in VX$ be a vertex such that  $U = G_x$. Then there is $r \geq 0$ such that $N_x = N_x^{[r]}$ and $N_x \not \leq N_x^{[r+1]}$. Thus for some vertex $y \in VX$ with $d(x, y)=r$, we have  $N_y \not \leq N_y^{[1]}$. Using that   $N$ is normal in $H$ and  that $H$ is vertex-transitive, we deduce that $N_x \not \leq N_x^{[1]}$. 

Coming back to the sequence $(N_k)$ from above, we deduce that, upon extracting, it  converges to a closed normal subgroup $N$ of $G$ with $N_x \not \leq N_x^{[1]}$. In particular $N \cap U \neq \{1\}$. Therefore $N$ is not contained in $\QZ(G)$, and it is thus cocompact in $G$. Invoking Lemma~\ref{lem-rU-coc-closed} again, we deduce that $N_k$ is cocompact in $G$ for all sufficiently large $k$, a contradiction. This proves the claim. 

The claim directly implies that for all $H \in \sub(G)$ sufficiently close to $G$, every closed normal subgroup of $H$ that is not cocompact intersects $U$ trivially, and is thus discrete. Thus such a group $H$ is quasi just-non-compact. 
\end{proof}

\subsection{Cocompact subgroups that are Chabauty limits of discrete subgroups}

Recall that a group $H$ is called \textbf{quasi-simple} if $H$ is perfect and $H/Z(H)$ is simple. For a finite group $G$, a \textbf{component} of $G$ is a subnormal subgroup that is quasi-simple. The \textbf{layer} of $G$, denoted by $E(G)$, is the subgroup generated by all quasi-simple subnormal subgroups.

The goal of this subsection is to establish the following. 

\begin{prop}\label{prop:CocoLimits-discrete}
Let $G$ be a compactly generated tdlc group and $U \leq G$ be a compact open subgroup with $\bigcap_{g \in G} g U g^{-1} =\{1\}$. Let $(\Gamma_k)$ be a sequence of discrete subgroups of $G$ converging to a cocompact subgroup $H \leq G$, and such that $U \cap \Gamma_k \neq \{1\}$ for all sufficiently large $k$. Assume that for every conjugate $V$ of $U$, we have $\QZ(V \cap H) = \{1\}$. Then the following assertions hold:
\begin{enumerate}[label=(\roman*)]
	\item $E(U \cap \Gamma_k) = \{1\}$ for all sufficiently large $k$. 
	\item There is a prime $p$ such that $H$ has an infinite compact locally normal subgroup that is pro-$p$.
\end{enumerate} 
\end{prop}

The following basic fact is of fundamental importance. It implies that the layer $E(G)$, which is a characteristic subgroup of $G$, is a perfect central extension of a direct product of non-abelian simple groups.

\begin{lem}\label{lem:ComponentsCommute}
Let $G$ be a finite group and $H \leq G$ be a subnormal subgroup. Given a component $L$ of $G$, either $L \leq H$ or $[L, H]=\{1\}$. In particular, any two distinct components centralize each other. 
\end{lem}
\begin{proof}
See \cite[(31.4) and (31.5)]{Aschbacher}. 
\end{proof}

 The \textbf{Fitting subgroup} of a group $G$, denoted by $F(G)$, is the characteristic subgroup generated by all nilpotent normal subgroups of $G$. If $G$ is finite, then $F(G)$ coincides with the direct product of all $O_p(G)$, where $p$ runs over all primes dividing the order of $G$.  We recall that $O_p(G)$ denotes the largest normal $p$-subgroup of the finite group $G$. In case $G$ is a profinite group, the same symbol denotes the largest closed normal subgroup of $G$ that is a pro-$p$ group.  The \textbf{generalized Fitting subgroup} of a finite group $G$, denoted by $F^*(G)$, is defined as $F^*(G) = E(G)F(G)$. Notice that the generalized Fitting subgroup of a non-trivial finite group is non-trivial. 
 %, and by $O^p(G)$ the smallest normal subgroup of $G$ affording a quotient that is a $p$-group. 

The relevance of those notions to our purposes is revealed by the following subsidiary facts. 

\begin{lem}\label{lem:Graph-subnormal}
	Let $X$ be a connected graph and let $G \leq \Aut(X)$. Let also $x, y \in VX$. Then for all $r > d(x, y)$, the group $G_x^{[r]}$ is a subnormal subgroup of $G_y$. 
\end{lem}
\begin{proof}
	Let $y = y_0, y_1, \dots, y_\ell = x$ be a shortest path from $y$ to $x$. We have
	$$G_y \trianglerighteq G_y^{[1]} \trianglerighteq G_{y, y_1}^{[1]} \trianglerighteq  \dots \trianglerighteq G_{y, y_1, \dots, y_{\ell-1}}^{[1]}.$$
	Notice that $G_x \geq   G_{y, y_1, \dots, y_{\ell-1}}^{[1]}$ and that  $G_x^{[r]}  \leq   G_{y, y_1, \dots, y_{\ell-1}}^{[1]}$ for all $r > \ell = d(x, y)$. Since $G_x^{[r]}$ is normal in $G_x$, we have $G_{y, y_1, \dots, y_{\ell-1}}^{[1]}\trianglerighteq G_x^{[r]}$. 
\end{proof}

\begin{lem}\label{lem:GeneralizedFitting-convergence}
Let $G$ be a compactly generated tdlc group and let $U \leq G$ be a compact open subgroup with $\bigcap_{g \in G} g U g^{-1}= \{1\}$. Let $(\Gamma_k)$ be a sequence of discrete subgroups of $G$ converging to a cocompact subgroup $H \leq G$. If $U \cap \Gamma_k \neq \{1\}$ for all sufficiently large $k$, then there is a conjugate $V$ of $U$ in $G$ which satisfies  at least one of the following properties:
\begin{enumerate}[label=(\roman*)]
	\item \label{item-lem-fitt-1} $E(V \cap \Gamma_k) \neq \{1\}$ for infinitely many $k$, and $\QZ(V \cap H)  \neq \{1\}$. 
	
	\item \label{item-lem-fitt-2} There is a prime $p$ such that $O_p(V \cap \Gamma_k) \neq \{1\}$ for infinitely many $k$, and $O_p(V \cap H) \neq \{1\}$.  
\end{enumerate}
\end{lem}

\begin{proof}
Let $X$ be a Cayley--Abels graph for $(G, U)$. Since $\bigcap_{g \in G} g U g^{-1}= \{1\}$, the $G$-action on $X$ is faitfhul. Since $H$ is cocompact in $G$, it has finitely many orbits of vertices on $X$.  It follows from Lemma~\ref{lem-rU-coc-closed} that $\Gamma_k$ is $(r, U)$-cocompact in $G$ for  all sufficiently large $k$, where $r = |H \backslash G/U|$.  

Let $x$ be the base vertex, i.e. the vertex corresponding to the trivial coset $U$ in $VX = G/U$. Let $R$ be such that the ball $B(x, R)$ contains a representative of the $J$-orbits of vertices on $X$ for all $(r, U)$-cocompact subgroups $J$ (see Lemma~\ref{lem-rU-cocom-uniform}).

Since $(\Gamma_k)_x \neq \{1\}$ for all sufficiently large $k$, we have $F^*((\Gamma_k)_x) \neq \{1\}$, so there is $\ell_k \geq 0$ such that  $F^*((\Gamma_k)_x) \leq  (\Gamma_k)_x^{[\ell_k]}$ and  $F^*((\Gamma_k)_x) \not \leq (\Gamma_k)_x^{[\ell_k +1]}$. We claim that there exist $y_k$ such that $F^*((\Gamma_k)_{y_k}) \not \leq (\Gamma_k)_{y_k}^{[2]}$. If $\ell_k \leq 1$, then we may take $y_k = x$. If $\ell \geq 2$, then we pick $y_k \in VX$ with $d(x, y_k) = \ell_k-1$ and $F^*((\Gamma_k)_x) \not \leq (\Gamma_k)_{y_k}^{[2]}$. By Lemma~\ref{lem:Graph-subnormal}, the group $(\Gamma_k)_x^{[\ell_k]}$ is subnormal in $(\Gamma_k)_{y_k}$, so that $F^*((\Gamma_k)_x^{[\ell_k]}) \leq F^*((\Gamma_k)_{y_k})$ . Since $F^*((\Gamma_k)_x) \leq  (\Gamma_k)_x^{[\ell_k]}$, we have $F^*((\Gamma_k)_x) = F^*((\Gamma_k)_x^{[\ell_k]})$,
hence $F^*((\Gamma_k)_x) \leq F^*((\Gamma_k)_{y_k})$. Therefore $F^*((\Gamma_k)_{y_k}) \not \leq (\Gamma_k)_{y_k}^{[2]}$, and we have proved the claim. Since $B(x, R)$ contains a set of representatives of the $\Gamma_k$-orbits of vertices, we may assume that $y_k \in B(x, R)$, and hence upon extracting we may assume that $y_k = y$ for some $y \in B(x, R)$ and all $k$. We set $V = G_y$, which is conjugate to $U = G_x$. Now we distinguish two cases.

Assume first that  $E((\Gamma_k)_y)  )  \not \leq (\Gamma_k)_y^{[2]}$ for infinitely many $k$. By definition of the layer, we may find a component $L_k$ of $(\Gamma_k)_y$ that is not contained in $(\Gamma_k)_y^{[ 2]}$, and upon extracting we may assume that $(L_k)$ converges to a subgroup $L \leq H_y$ with $L \not \leq H_y^{[ 2]}$. By Lemma~\ref{lem:ComponentsCommute}, the component $L_k$ commutes with $(\Gamma_k)_y^{[ 2]}$, so it follows that $L$ commutes with $H_y^{[ 2]} = \lim_k (\Gamma_k)_y^{[2]}$ (recalling that taking an intersection with a fixed open subgroup is a Chabauty continuous operator on the $\sub(G)$). Thus the centralizer of $L \leq H_y$  in $H_y$ contains $H_y^{[2]} $, and is thus open.  It follows that $\QZ(V \cap H) =\QZ(H_y) \neq \{1\}$. Thus the case (i) of the statement occurs. 

Assume in a second case that  $E((\Gamma_k)_y)  )   \leq (\Gamma_k)_y^{[2]}$ for all but finitely many $k$. Then we have $F((\Gamma_k)_y) \not \leq (\Gamma_k)_y^{[2]}$ all but finitely many  $k$ since $ E((\Gamma_k)_y) F((\Gamma_k)_y) = F^*((\Gamma_k)_y) \not \leq (\Gamma_k)_y^{[2]}$.  Since $(\Gamma_k)_y$ acts faithfully on the regular graph $X$ fixing the vertex $y$, it follows that every prime dividing the order of $(\Gamma_k)_y$  is at most the degree of $X$. Therefore, there must exist a prime $p$ such that $O_p((\Gamma_k)_y)\not \leq (\Gamma_k)_y^{[2]}$  for  infinitely many  $k$. Upon extracting, the sequence $O_p((\Gamma_k)_y)$ converges to a closed normal pro-$p$-subgroup of $H_y$ that is not contained in $H_y^{[2]}$. In particular $O_p(H_y) \not \leq \{1\}$, so that  the case (ii) of the statement occurs. 
\end{proof}	

\begin{proof}[Proof of Proposition~\ref{prop:CocoLimits-discrete}]
We apply Lemma~\ref{lem:GeneralizedFitting-convergence}. Since the conclusion \ref{item-lem-fitt-1} of the lemma is ruled out by our assumption, it follows that the conclusion \ref{item-lem-fitt-2} must hold. Therefore there is a conjugate $V$ of $U$ in $G$ such that $E(\Gamma_k \cap V)=\{1\}$ for all sufficiently large $k$, and that $K = O_p(V \cap H)$ is non-trivial for some prime $p$. Thus $K$ is a locally normal subgroup of $H$ that is pro-$p$. If $K$ is finite then $K$ must lie inside $\QZ(V \cap H)$, which is absurd because $\QZ(V \cap H)$ is trivial. So $K$ is infinite, and the statement holds.
\end{proof}

\subsection{Approximations of quasi just-non-compact groups by discrete subgroups}

\begin{thm} \label{thm-limit-lattices-pro-p}
Let $L$ be a compactly generated tdlc group admitting a compact open subgroup $U$ such that $\bigcap_{l \in L} l U l^{-1}= \{1\}$. Assume that there exists a closed cocompact subgroup $G \leq L$ that is quasi just-non-compact and non-discrete, and a sequence of discrete subgroups of $L$ that Chabauty converges to a finite index open subgroup $H \leq G$. Then there is a prime $p$ and a compact open subgroup of $G$ that is a pro-$p$ group.
\end{thm}

\begin{proof}
By Proposition~\ref{prop-carac-qjnc}, the quasi-center $\QZ(G)$ is discrete, so without loss of generality we may assume that $U \cap \QZ(G)=\{1\}$. Also since $\QZ(G)$ is normal in $G$ and $G$ is cocompact in $L$, upon passing to an open subgroup of $U$ we may also assume that $\QZ(G)$ intersects trivially all $L$-conjugates of $U$ (by replacing $U$ by $\cap_{i=1}^r l_i U l_i^{-1}$, where $l_1 U, \ldots, l_r U$ are representatives for the $G$-orbits in $L/U$). Hence whenever $V$ is an $L$-conjugate of $U$, we have $\QZ(V \cap H)  \leq V \cap \QZ(G) = 1$ since $H$ in an open subgroup of $G$. Since $H$ is non-discrete, the discrete groups $\Gamma_k$ intersect $U$ non-trivially for large enough $k$. Therefore it follows from Proposition~\ref{prop:CocoLimits-discrete} that there is a prime $p$ such that $H$ has an infinite compact locally normal subgroup $K$ that is pro-$p$. Since $H$ is open in $G$, that subgroup $K$ is also locally normal in $G$. Note that $(K \cap U ) \cap \QZ(G) \leq U \cap \QZ(G) = 1$, so the quotient group $G/\QZ(G)$ also has an infinite compact subgroup that is locally normal and pro-$p$. By Proposition~\ref{prop-carac-qjnc} $G/\QZ(G)$ is just-non-compact, and it is also compactly generated, so by Proposition~\ref{prop:JNC:primeContent} the group $G/\QZ(G)$ must have an open pro-$p$ subgroup. It follows that the same is true in $G$ since $\QZ(G)$ is a discrete subgroup of $G$.
\end{proof}

We now complete the proof of Theorem~\ref{thmintro:Zassenhaus} from the introduction. 

\begin{proof}[Proof of Theorem~\ref{thmintro:Zassenhaus}]
We let $G$ be a non-discrete, compactly generated, quasi just-non-compact tdlc group, and $H$ an open subgroup of finite index in $G$ that is a Chabauty limit of a sequence of discrete subgroups of $G$. Since $G$ is quasi just-non-compact, every compact normal subgroup of $G$ is finite. In particular the group $G$ admits a compact open subgroup $U$ such that $\bigcap_{g \in G} g U g^{-1}= \{1\}$. Hence we may apply Theorem \ref{thm-limit-lattices-pro-p} with $L = G$, and we deduce that there exist a prime $p$ and an open pro-$p$ subgroup $V$ of $G$. We have to show that $V$ is not topologically finitely generated. Argue by contradiction and assume that $V$ is topologically finitely generated. Then so is $W = V \cap H$. Since $W$ is also pro-$p$, it follows that $W$ has an open Frattini subgroup \cite[Proposition 1.14]{anprop}, and hence $W$ is an isolated point of $\sub(W)$ \cite[Theorem~5.6]{GartsideSmith}. Since $(\Gamma_k)$ converges to $H$, the sequence of intersections $(\Gamma_k \cap W)$ converges to $W$, so that $\Gamma_k$ contains the open subgroup $W$ for all sufficiently large $k$, contradicting that $\Gamma_k$ is discrete.
\end{proof}

Combining Theorem~\ref{thmintro:Zassenhaus} with Proposition~\ref{prop:ChabNeighb-qjnc}, the following result is immediate. 

\begin{cor}\label{cor:ChabNeighb-qjnc}
	Let $G$ be a non-discrete, compactly generated, quasi just-non-compact tdlc group. Assume that at least one of the following conditions is satisfied:
	\begin{enumerate}[label=(\arabic*)]
		\item $G$ has a compact open subgroup that is topologically finitely generated. 
		
		\item No compact open subgroup is pro-$p$ for any prime $p$. 
	\end{enumerate}
Then the collection of non-discrete, compactly generated, quasi just-non-compact closed subgroups $H \leq G$ forms a neighbourhood of $G$  in $\sub(G)$.  
\end{cor}	

\begin{rmq}	
That result is of special interest  when the group $G$ is topologically simple. If $G$ is not an isolated point in $\sub(G)$ and if $G$ satisfies (1) or (2), then $G$ is Chabauty approximated by non-discrete, compactly generated,  quasi just-non-compact groups. Those are subjected to Proposition~\ref{prop-carac-qjnc} and \cite[Theorem~E]{CaMo-decomp}, hence involve quasi products of compactly generated simple groups. This might be a tool to construct new compactly generated simple groups from known ones.
\end{rmq}

\section{Irreducibility conditions for lattices in products} \label{sec-irr}

\subsection{Cocompact closed subgroups in products}

The proof of \ref{it:irre3} $\Rightarrow$ \ref{it:irre1'} in Theorem~\ref{thm:irred:JNC} is based on Propositions \ref{prop-no-inter-implies-discrete} and \ref{prop-H=G}, which are of independent interest.  

\begin{prop} \label{prop-no-inter-implies-discrete}
Let $G_1$ be a tdlc group, $G_2$ be compactly generated locally compact group with a discrete polycompact radical, and $H \leq G = G_1 \times G_2$ a closed cocompact subgroup. If $H \cap G_2$ is a discrete subgroup of $G_2$, then $H \cap G_1$ is open in $H$.
\end{prop}

\begin{proof}
Let $U_1$ be a compact open subgroup of $G_1$, and let $O = H \cap (U_{1} \times G_2)$. Note that $O$ is compactly generated as it is cocompact in the compactly generated group $U_{1} \times G_2$. We write $N = H \cap G_2$, which is contained in $O$. The group $O/N$ embeds continuously in the profinite group $U_{1}$, and hence $O/N$ is residually discrete. Since $N$ is discrete by assumption, Proposition \ref{prop:CapMon-RD} together with Proposition \ref{prop-exten-RD} imply that the group $O$ admits a compact open normal subgroup $K$. Now according to Proposition \ref{prop:CocoNormal} the subgroup $K$ is contained in a compact normal subgroup of $U_{1} \times G_2$ since $K$ has a cocompact normalizer in $U_{1} \times G_2$. But $G_2$ has a discrete polycompact radical, so every compact normal subgroup of $U_{1} \times G_2$ has a finite projection to the factor $G_2$. In particular there is a finite index closed subgroup of $K$ that is contained in $U_1$, and $H \cap G_1$ is open in $H$.
\end{proof}

%\begin{rmq} \label{rmq-(T)-0}
%Proposition \ref{prop-no-inter-implies-discrete} remains true when $H$ is a closed subgroup of finite covolume in $G$, provided that $G_2$ has property (T) and $G_2$ has a discrete amenable radical. The proof follows the same lines, but we invoke on the one hand property (T) to ensure that $O$ is compactly generated, and on the other hand the main result of \cite{BDL_URS} instead of Proposition \ref{prop:CocoNormal} in order to derive that $K$ has a finite projection to $G_2$.
%\end{rmq}

\begin{prop} \label{prop-(T)-0}
Let $G_1,G_2$ be tdlc groups, and suppose that $G_2$ has property (T) and $G_2$ has a discrete amenable radical. If $H$ is a closed subgroup of finite covolume in $G$, then the conclusion of Proposition \ref{prop-no-inter-implies-discrete} holds.
\end{prop}

\begin{proof}
The proof follows the same lines as Proposition \ref{prop-no-inter-implies-discrete}. The subgroup $O = H \cap (U_{1} \times G_2)$ has finite covolume in $U_{1} \times G_2$, and hence also has property (T). In particular $O$ is compactly generated. By the same argument as above, we see that $O$ admits a compact open normal subgroup $K$, so that $p_2(K)$ is a compact normal subgroup of $p_2(O)$. Now since $p_2(O)$ has finite covolume in $G_2$, the amenable radical of $p_2(O)$ must be contained in the amenable radical of $G_2$ (by \cite[Proposition 4.4]{Furst-bourb} and \cite[Proposition 7]{Furman-min-strg}; or by the main result of \cite{BDL_URS}). Since $G_2$ has a discrete amenable radical, the subgroup $p_2(K)$ is therefore contained in a discrete subgroup of $G_2$. So $K$ has a finite projection to $G_2$, and $H \cap G_1$ is open in $H$.
\end{proof}

\begin{cor} \label{cor-no-inter-factor-implies-discrete}
Let $G_1,\ldots,G_n$ be compactly generated tdlc groups with $W(G_i)$ discrete for every $i$. Let $H \leq G = G_1 \times \ldots \times G_n$ be a closed cocompact subgroup such that $H \cap G_i$ is discrete in $G_i$ for all $i$. Then $H$ is discrete.
\end{cor}

\begin{proof}
Since the polycompact radical of a finite product is the product of the polycompact radicals, every subproduct of $G$ has a discrete polycompact radical, so we deduce that $\Ker(p_i|_H)$ is open in $H$ for every $i$ according to Proposition \ref{prop-no-inter-implies-discrete}. Therefore $\bigcap_i \Ker(p_i|_H) = 1$ remains open in $H$, and hence $H$ is discrete.
\end{proof}

\begin{lem}\label{lem:ConstantOrbits}
Let $G = G_1 \times G_2$ be a product of tdlc groups and $H \leq G$ be a closed subgroup such that $H \cap G_2$ is cocompact in $G_2$. For each compact open subgroup $U_2$ of $G_2$, there is a compact open subgroup $U_1$ of $G_1$ such that for each open subgroup $V_1 \leq U_1$, the projections 
$p_2(H \cap (V_1 \times G_2))$  and $p_2(H \cap (U_1 \times G_2))$
have the same orbits on $G_2/U_2$.
\end{lem}

\begin{proof}
For each compact open subgroup $W \leq G_1$, the intersection $H\cap (W \times G_2)$ contains $H \cap G_2$. Thus the number of orbits on $G_2/U_2$ of the  projection $p_2(H\cap (W \times G_2))$ is bounded above by the number of $(H \cap G_2)$-orbits, which is finite in view of the hypothesis that $H \cap G_2$ is cocompact in $G_2$. Let $M$ be the maximum of the number of orbits over all possible $W$, and let $U_1 \leq G_1$ be compact open subgroup such that $p_2(H\cap (U_1 \times G_2))$ has exactly $M$ orbits on $G_2/U_2$. Then for every open subgroup $V_1 \leq U_1$, the projection $p_2(H\cap (V_1 \times  G_2))$ has also $M$ orbits on $G_2/U_2$. Since $H\cap (V_1 \times G_2)$ is a subgroup of $H\cap (U_1 \times G_2)$, it follows that $p_2(H\cap (V_1 \times G_2))$ and  $p_2(H\cap (U_1 \times G_2))$ must have the same orbits on $G_2/U_2$.
\end{proof}

\begin{prop}\label{prop:OpenProj}
Let $G = G_1 \times G_2$ be a product of compactly generated tdlc groups and $H \leq G$ be a closed  subgroup. Assume that   $p_1(H)$ is dense in $G_1$.
\begin{enumerate}[label=(\roman*)]
	\item If $H \cap G_2$ is cocompact in $G_2$, then   for  each compact open subgroup $U_2$ of $G_2$, the projection $p_1(H \cap (G_1 \times U_2))$ is an  open subgroup of $G_1$. 
	
	\item If there is a compact open subgroup $U_2$ of $G_2$ such that  $(H \cap G_2)U_2 = G_2$, then  $p_1(H \cap (G_1 \times U_2)) = G_1$.
\end{enumerate}

\end{prop}

\begin{proof}
(i) We fix a compact open subgroup $U_1 \leq G_1$ afforded by Lemma~\ref{lem:ConstantOrbits}. 
Let also $u \in U_1$ and $V_1 \leq U_1$ be any open subgroup of $U_1$. Since $p_1(H)$ is dense in $G_1$, there exists $h \in H$ with $p_1(h) \in V_1 u \subset U_1$. In particular, we have $h \in H \cap (U_1 \times G_2)$. By Lemma~\ref{lem:ConstantOrbits}, there exists $h' \in H \cap (V_1 \times G_2)$ such that $h''=h' h \in G_1 \times U_2$. We have $h'' \in H \cap (G_1 \times U_2)$ and $p_1(h'')  = p_1(h')p_1(h) \in V_1 u$. Since $V_1$ was arbitrary, we infer that the closure of $p_1( H \cap (G_1 \times U_2))$ contains $U_1$, and is thus open in $G_1$. But the subgroup $U_2$ being compact, the projection $p_1( H \cap (G_1 \times U_2))$ is closed, so we actually deduce that $p_1( H \cap (G_1 \times U_2))$ is open, as required.

\medskip \noindent (ii)
 By hypothesis  $p_1(H)$ is dense in $G_1$ and by (i),   $p_1(H)$ is open in $G_1$. Thus $p_1(H)= G_1$. Let now $g \in G_1$ and choose $h \in H$ with $p_1(h) = g$. Since   $(H \cap G_2)U_2 = G_2$, there exists $h' \in H \cap G_2$ with $h'p_2(h) = p_2(h'h) \in U_2$. Hence $h'h \in H \cap (G_1 \times U_2)$. Since $p_1(h') = 1$, we have $p_1(h'h) = p_1(h) = g$. 
\end{proof}

The following result can be compared with  \cite[Proposition~2.2]{BuMo_Wang}, which is concerned with closed cocompact subgroups of a product of two locally quasi-primitive groups of automorphisms of locally finite trees.

\begin{prop} \label{prop-H=G}
Let $G_1,\ldots,G_n$ be compactly generated tdlc groups that are quasi just-non-compact. Let $H \leq G = G_1 \times \ldots \times G_n$ be a closed  subgroup, such that the image of $p_i : H \rightarrow G_i$ is dense for every $i$; and the image of $p_\Sigma : H \rightarrow G_\Sigma$ is non-discrete for every non-empty $\Sigma \subseteq \left\{1,\ldots,n\right\}$. Assume that at least one of the following conditions is satisfied:
\begin{enumerate}[label=(\arabic*)]
	\item $H$ is cocompact in $G$. 
	
	\item $H$ is of finite covolume in $G$, and $G$ has Kazhdan's property (T). 
\end{enumerate}
Then the following hold:
\begin{enumerate}[label=(\roman*)]
	\item \label{item-H-contains-prod} $H$ contains $K_1 \times \ldots \times K_n$ for some closed cocompact normal subgroups $K_i \triangleleft G_i$. In particular $H$ is cocompact in $G$. 
	
	\item \label{item-H-proj-open} For every $i$ and every compact open subgroup $U_1 \times \ldots \times U_n \leq G$, the projection $p_i(H \cap (U_1 \times \ldots \times G_i \times \ldots \times U_n))$ is a finite index open subgroup of $G_i$.  
\end{enumerate}
\end{prop}

\begin{proof}
For every $i$ we  write $K_i = H \cap G_i$. The subgroup $K_i$ is centralized by $\prod_{j \neq i} G_j$ and normalized by $H$. So the normalizer of $K_i$ in $G$, which is a closed subgroup, contains the subgroup $(\prod_{j \neq i} G_j) H$, which is dense in $G$ by assumption. Therefore $K_i$ is normal in $G$, and hence is either discrete or cocompact in $G_i$ since $G_i$ is quasi just-non-compact.

We let $\Pi \subseteq \left\{1,\ldots,n\right\}$ be the set of $i$ such that $K_i$ is not discrete, and $\Sigma$ be the complement of $\Pi$. We have to show that $\Pi = \left\{1,\ldots,n\right\}$ and $\Sigma  = \emptyset$. The group $K_\Pi = \prod_{i \in \Pi} K_i$ is cocompact in $G_\Pi$, so the projection $p_\Sigma : G \rightarrow G_\Sigma$ factors through a proper map $G / K_\Pi \rightarrow G_\Sigma$. Moreover $K_\Pi$ is contained in $H$, so that the quotient $H / K_\Pi$ is a closed   subgroup of $G / K_\Pi$. Since   map map $G / K_\Pi \rightarrow G_\Sigma$ is proper,   it follows that the projection of $H / K_\Pi$ to $G_\Sigma$, which is also the projection of $H$ to $G_\Sigma$, has a closed image.

If (1) holds, then we may invoke  Proposition~\ref{prop-no-inter-implies-discrete} and deduce that  $\Ker(p_j|_H)$ is open in $H$ for every $j \notin \Pi$. If (2) holds, then we observe that for every $i$, the amenable radical of $G_i$ is discrete, since it cannot be cocompact by the hypothesis that $G_i$ is a non-compact group with property~(T). In particular the amenable radical of $\prod_{i \neq j} G_i$ is discrete for all $j$. Therefore, by Proposition~\ref{prop-(T)-0}, the conclusion of Proposition~\ref{prop-no-inter-implies-discrete} also holds in that case, and we deduce again that for every $j \notin \Pi$, the subgroup $\Ker(p_j|_H)$ is open in $H$. 

It follows that $H \cap G_\Pi = \bigcap_{j \notin \Pi} \Ker(p_j|_H)$ is an open subgroup of $H$. Since projection $H \to G_\Sigma$ has closed image and factors through the discrete group $H/H \cap G_\Pi$, we infer   that the projection of $H$ to $G_\Sigma$ has discrete image. In view of the hypotheses, this implies that $\Sigma$ is the empty set.  Therefore $\Pi = \left\{1,\ldots,n\right\}$, and the proof of \ref{item-H-contains-prod} is complete.

In order to prove \ref{item-H-proj-open}, we fix $i$ and write $L_i = p_i(H \cap (U_1 \times \ldots \times G_i \times \ldots \times U_n))$. According to \ref{item-H-contains-prod} we have that $H \cap \prod_{j \neq i} G_j$ is cocompact in $\prod_{j \neq i} G_j$, so we may apply Proposition \ref{prop:OpenProj}, which says that $L_i$ is an open subgroup of $G_i$. But by \ref{item-H-contains-prod} again we also know that $L_i$ is cocompact in $G_i$, and hence $L_i$ if a finite index subgroup of $G_i$.
\end{proof}

\begin{cor} \label{cor-3-1-JNC}
Let $G_1,\ldots,G_n$ be compactly generated, quasi just-non-compact tdlc groups. Let $\Gamma \leq G = G_1 \times \ldots \times G_n$ be a cocompact lattice such that $p_i(\Gamma)$ is dense in $G_i$ for every $i=1,\ldots,n$ and satisfying \ref{it:irre3}. Then for every $\Sigma \subsetneq \left\{1,\ldots,n\right\}$, the closure of the image of $p_\Sigma \colon \Gamma \to G_\Sigma$ contains $\prod_{i \in \Sigma} G_i^{(\infty)}$.
\end{cor}

\begin{proof}
First note that the condition \ref{it:irre3} implies that all factors $G_i$ are non-discrete. Fix $\Sigma \subsetneq \left\{1,\ldots,n\right\}$, and write $H = \overline{p_\Sigma(\Gamma)}$. Note that there is nothing to prove if $\Sigma$ is a singleton, so we assume that $\Sigma$ has cardinality at least two. The subgroup $H$ is closed and cocompact in $G_\Sigma$, and has a dense projection on each $G_i$ for $i \in \Sigma$. Since $\Gamma$ satisfies \ref{it:irre3}, Proposition \ref{prop-H=G} can be applied to the group $H$ inside $G_\Sigma$. So we deduce that $H$ contains $\prod_{i \in \Sigma} K_i$ for some closed cocompact normal subgroups $K_i \leq G_i$. By Proposition \ref{prop-carac-qjnc} we must have $G_i^{(\infty)} \leq K_i$, whence the statement. 
\end{proof}

We end this paragraph with an application to commensurators of lattices. Recall that an irreducible lattice $\Gamma$ in a semi-simple Lie group $G$ with trivial center and no compact factor has a commensurator $\comm_G(\Gamma)$ that is either discrete or dense in $G$, and $\comm_G(\Gamma)$ is dense if and only if $\Gamma$ is an arithmetic lattice \cite[Theorem IX.B]{Margulis}. The following statement shows that a similar \enquote{discrete or dense} dichotomy holds in the setting of products of quasi just-non-compact groups.

\begin{cor}
Let $G_1,\ldots,G_n$ be compactly generated, quasi just-non-compact tdlc groups, and $G = G_1 \times \ldots \times G_n$. Assume that at least $n-1$ factors satisfy $G_i = G_i^{\infty}$. Let $\Gamma \leq G$ be a lattice such that $p_i(\Gamma)$ is dense in $G_i$ for every $i=1,\ldots,n$ and satisfying \ref{it:irre3}. Assume also that $\Gamma$ is cocompact, or that $G$ has Kazhdan's property (T). Then $\comm_G(\Gamma)$ is either discrete or dense in $G$.
\end{cor}

\begin{proof}
We let $H$ be the closure of $\comm_G(\Gamma)$ in $G$. We have to show that either $H$ is discrete or $H = G$. Assume that $H$ is not discrete. Then we may apply Proposition \ref{prop-H=G}, and we deduce that $H$ contains $K_1 \times \ldots \times K_n$ for some closed cocompact normal subgroups $K_i \triangleleft G_i$. That $G_i = G_i^{\infty}$ is equivalent to saying that $G_i$ has no proper cocompact closed normal subgroup, so it follows that $H$ contains $G_i$ for all but possibly one factor. But since $H$ is a closed subgroup of $G$ and $H$ has a dense projection to this remaining factor, we deduce that $H=G$ as required. 
\end{proof}

\subsection{Relations  between the irreducibility conditions}\label{sec:Proofs}

\begin{lem} \label{lem-implic-banal}
Let $G_1,\ldots,G_n$ be locally compact groups that are all non-compact. For a lattice $\Gamma \leq G = G_1 \times \ldots \times G_n$, we have \ref{it:irre2} $\Rightarrow$ \ref{it:irre3}, \ref{it:irre2} $\Rightarrow$ \ref{it:irre0}, and \ref{it:irre4} $\Rightarrow$ \ref{it:irre0}.
\end{lem}

\begin{proof}
If \ref{it:irre3} fails, there exists a partition $\Sigma \cup \Pi = \left\{1,\ldots,n\right\}$ with $\Pi, \Sigma \neq \varnothing$, such that $p_\Sigma(\Gamma)$ is discrete in $G_\Sigma$. Since $\overline{p_\Sigma(\Gamma)}$ is also a lattice in $G_\Sigma$, it follows that $p_\Sigma$ maps $\Gamma$ onto a lattice in $G_\Sigma$. It then follows from \cite[Theorem~I.1.13]{Raghu} that $\Ker(p_\Sigma|_\Gamma) = \Gamma \cap G_\Pi$ is a lattice in $G_\Pi$, which is non-compact by assumption. So $\Gamma \cap G_\Pi$ is non-trivial, which contradicts \ref{it:irre2}.

If \ref{it:irre0} fails, i.e.\ if $(G_\Pi \cap \Gamma)(G_\Sigma \cap \Gamma)$ is of finite index in $\Gamma$ for some partition $\Pi \cup \Sigma = \left\{1,\ldots,n\right\}$ with $\Pi, \Sigma \neq \varnothing$, then arguing as above we see that $G_\Pi \cap \Gamma$ is a lattice in $G_\Pi$, and obtain again a contradiction with \ref{it:irre2}. 

If  \ref{it:irre0} fails, then  there exists a partition $ \Pi \cup \Sigma = \left\{1,\ldots,n\right\}$ with $\Pi, \Sigma \neq \varnothing$ such that the direct product  $\Gamma_1 = (G_\Pi \cap \Gamma)(G_\Sigma \cap \Gamma)$ is of finite index in $\Gamma$. In particular $\Gamma_1$ is a lattice in $G = G_\Pi \times G_\Sigma$. It follows that $ G_\Pi \cap \Gamma$ is a lattice in $G_\Pi$ and $ G_\Sigma \cap \Gamma$ is a lattice in $G_\Sigma$. Thus both factors in the direct product decomposition of $\Gamma_1$ are non-trivial, so that \ref{it:irre4} fails.
\end{proof}

\begin{lem} \label{lem-1-2}
Let $G_1,\ldots,G_n$ be locally compact groups, and $\Gamma \leq G = G_1 \times \ldots \times G_n$ a discrete subgroup. Let $\Sigma \cup \Pi = \left\{1,\ldots,n\right\}$ be a partition. 

\begin{enumerate}[label=(\roman*)]
	\item Assume that   $\QZ(G_i) = 1$ for every $i \in \Pi$. If  $p_{\Pi}: \Gamma \rightarrow G_{\Pi}$ has dense image, then $\Gamma \cap G_{\Pi}$ is trivial. 	
	In particular if we have $\QZ(G_i) = 1$ for every $i$, then \ref{it:irre1} $\Rightarrow$ \ref{it:irre2}.
	
	\item Assume that $G_i$ is monolithic for all $i \in \Pi$, with $\QZ(\Mon(G_i)) = \{1\}$. If the closure of the image of $p_{\Pi}: \Gamma \rightarrow G_{\Pi}$ contains $G_\Pi^+ = \prod_{i \in \Pi} \Mon(G_i)$, then $\Gamma \cap G_{\Pi}$ is trivial. 	In particular, if $G_i$ is monolithic  with $\QZ(\Mon(G_i)) = \{1\}$ for all $i$, then  \ref{it:irre1'} $\Rightarrow$ \ref{it:irre2}.
\end{enumerate}
\end{lem}

\begin{proof}
(i) The subgroup $\Gamma \cap G_{\Pi}$ is normalized by $\Gamma$, and   centralized by $G_{\Sigma}$. Since the subgroup $G_{\Sigma} \Gamma$ is dense in $G$ by our assumption on the projection of $\Gamma$ on $G_{\Pi}$, it follows that $\Gamma \cap G_{\Pi}$ is a discrete normal subgroup of $G$. Therefore it lies in the quasi-center $\QZ(G_{\Pi})$. Notice that the quasi-center of a product group is the product of their quasi-centers (see \cite[Lemma~5.5]{CapWes}). So $\QZ(G_{\Pi})$ is trivial, and so is $\Gamma \cap G_{\Pi}$.

\medskip \noindent (ii) Let $H = \overline{p_\Pi(\Gamma)}$. We have $H \geq G_\Pi^+ $ by hypothesis. The group $N = \Gamma \cap G_\Pi$ is a discrete subgroup of $H$ normalized by $\Gamma$, hence by $\overline{p_\Pi(\Gamma)}$. Thus it is normal in $H$, hence contained in $\QZ(H)$. Again $G_\Pi^+$ has trivial quasi-center since all $\Mon(G_i)$ have this property for $i \in \Pi$, so 
$$N \cap G_\Pi^+  \leq \QZ(H) \cap G_\Pi^+ \leq \QZ(G_\Pi^+)= \{1\}$$ 
since $H \geq G_\Pi^+ $. Thus $N$ and $G_\Pi^+ $ are normal subgroups of $H$ with a trivial intersection, hence they commute. On the other hand we have $\mathrm C_{G_\Pi}(G_\Pi^+) = \prod_{i \in \Pi} \mathrm C_{G_i}(\Mon(G_i))$. Observe that   $\mathrm C_{G_i}(\Mon(G_i))$ must be trivial, since otherwise $\Mon(G_i)$ would be abelian, and hence equal to its quasi-center, which contradicts the hypothesis that $\QZ(\Mon(G_i)) = \{1\}$. We deduce that $N$ is trivial, which is the required conclusion. 
\end{proof}

\begin{prop} \label{prop-5-3}
Let $G_1,\ldots,G_n$ be compactly generated tdlc groups, and $\Gamma \leq G = G_1 \times \ldots \times G_n$ a cocompact lattice. Let $\Sigma \cup \Pi = \left\{1,\ldots,n\right\}$ be a partition such that $W(G_i)$ is discrete for every $i \in \Pi$ and such that $p_{\Sigma}: \Gamma \rightarrow G_{\Sigma}$ has discrete image. Then $p_\Pi(\Gamma)$ is discrete and $(G_\Pi \cap \Gamma)(G_\Sigma \cap \Gamma)$ has finite index in $\Gamma$.

In particular if $W(G_i)$ is discrete for every $i$ and $\Gamma$ is a cocompact lattice, then \ref{it:irre0} $\Rightarrow$ \ref{it:irre3}.
\end{prop}

\begin{proof}
Since $\overline{p_\Sigma(\Gamma)}$ is   cocompact in $G_\Sigma$, it follows that $p_\Sigma$ maps $\Gamma$ onto a cocompact lattice in $G_\Sigma$. It then follows from \cite[Theorem~I.1.13]{Raghu} that $\Ker(p_{\Sigma}|_\Gamma) = G_\Pi \cap \Gamma$ is a cocompact lattice in $G_\Pi$. Notice that $G_\Pi \cap \Gamma$ is a discrete subgroup of $G_\Pi$ normalized by $H = \overline{p_\Pi(\Gamma)}$. Since having a discrete polycompact radical is stable under taking finite direct products, from Corollary \ref{cor:NormalizerLattice} we infer that $H$ is discrete. In particular $p_\Pi(\Gamma)$ is discrete.  Using again \cite[Theorem~I.1.13]{Raghu} it now follows  $\Ker(p_{\Pi}|_\Gamma) = G_\Sigma \cap \Gamma$ is a cocompact lattice in $G_\Sigma$.   Thus $\Gamma_1=  (G_\Pi \cap \Gamma)(G_\Sigma \cap \Gamma)$  is a cocompact lattice in $G = G_\Pi \times G_\Sigma$. The index $|\Gamma : \Gamma_1|$ is thus finite.
\end{proof}

\begin{rmq} \label{rmq-(T)-2}
In Proposition \ref{prop-5-3} the assumption that $\Gamma$ is cocompact in $G$ can be replaced by the assumptions that $\Gamma$ is finitely generated and all groups $G_i$ have discrete amenable radical, by invoking \cite[Corollary~5.4]{CapWes} instead of Corollary \ref{cor:NormalizerLattice}. In particular it is for instance enough that all groups $G_i$ are quasi just-non-compact with property (T).
\end{rmq}

\begin{cor}\label{cor:ReductionToIrr}
Let $G_1,\ldots,G_n$ be compactly generated tdlc groups with discrete polycompact radical, and  $\Gamma \leq G = G_1 \times \ldots \times G_n$ be a cocompact lattice such that $p_i \colon \Gamma \rightarrow G_i$ has dense image for every $i=1,\ldots,n$.  

Then there is a partition $\Pi_{1} \cup \dots \cup \Pi_{\ell}$ of the set $\{1, \dots, n\}$ such that for all $j = 1, \dots, \ell$, the projection $\Gamma_{\Pi_j}  =  p_{\Pi_j}(\Gamma) \leq G_{\Pi_j}$ is  a cocompact lattice in $G_{\Pi_j}$ satisfying \ref{it:irre3} and \ref{it:irre0}, and $\Gamma$ is contained as a finite index subgroup of the product $\prod_{j=1}^\ell  \Gamma_{\Pi_j}$. 
\end{cor}

\begin{proof}
By  Proposition~\ref{prop-5-3}, for every partition $\Pi \cup \Sigma$ of the set $\{1, \dots, n\}$, if the projection $p_\Pi(\Gamma)$ is discrete, then $p_\Sigma(\Gamma)$ is discrete as well. Using this and a straightforward induction, we deduce that there is a partition $\Pi_1\cup \dots \cup \Pi_\ell$ of $\{1, \dots, n\}$ such that $p_{\Pi_i}(\Gamma)$ is discrete, and  $p_{\Sigma_i}(\Gamma)$ is non-discrete for all $i$, and  all $\Sigma_i \subsetneq \Pi_i$. Now, for all $i$, by construction, the projection $p_{\Pi_i}(\Gamma)$ is a lattice in $G_{\Pi_i}$ that satisfies \ref{it:irre3} (and hence also \ref{it:irre0}). It is clear that $\Gamma$ is contained in $\prod_{i=1}^\ell p_{\Pi_i}(\Gamma)$. Since the latter is a lattice in $G$ and since $\Gamma$ is also a lattice, it follows that the index of $\Gamma $ in $\prod_{i=1}^\ell p_{\Pi_i}(\Gamma)$ is finite.
\end{proof}

\begin{rmq} \label{rmq-(T)-3}
As before, Corollary~\ref{cor:ReductionToIrr} is also valid for non-uniform lattices provided all groups $G_i$ are quasi just-non-compact with property (T). The small modifications required in the proof are indicated in Remark~\ref{rmq-(T)-2}. 
\end{rmq}

We are now ready to complete the proof of Theorem~\ref{thm:irred:JNC}

\begin{proof}[Proof of Theorem~\ref{thm:irred:JNC}]
We assume henceforth that $G_1,\ldots,G_n$ are non-discrete compactly generated quasi just-non-compact groups, and $\Gamma \leq G = G_1 \times \ldots \times G_n$ is a cocompact lattice such that $p_i \colon \Gamma \rightarrow G_i$ has dense image for every $i=1,\ldots,n$. Recall that by Proposition \ref{prop-carac-qjnc} for all $i$ we have that $\QZ(G_i)$ is discrete in $G_i$, $G_i^{(\infty)}$ is a non-discrete cocompact subgroup of $G_i$ and every normal subgroup of $G_i$ is either contained in $\QZ(G_i)$ or contains $G_i^{(\infty)}$. 

\medskip \noindent
\ref{it:irre2} $\Rightarrow$ \ref{it:irre0} and \ref{it:irre4} $\Rightarrow$ \ref{it:irre0} follow from Lemma \ref{lem-implic-banal}.

\medskip \noindent
\ref{it:irre0} $\Rightarrow$ \ref{it:irre3} is consequence of Proposition \ref{prop-5-3} since we have $W(G_i) \leq \QZ(G_i)$ for all $i$.

\medskip \noindent
\ref{it:irre3} $\Rightarrow$ \ref{it:irre1'} follows from Corollary~ \ref{cor-3-1-JNC}.

So in order to complete the proof of statement \ref{item-thmirred-qJNC} of Theorem~\ref{thm:irred:JNC}, we need to show \ref{it:irre1'} $\Rightarrow$ \ref{it:irre0}. If \ref{it:irre0} fails then there is a partition $\Pi \cup \Sigma = \left\{1,\ldots,n\right\}$ with $\Pi, \Sigma \neq \varnothing$ such that $(G_\Pi \cap \Gamma)(G_\Sigma \cap \Gamma)$ has finite index in $\Gamma$. In particular $p_\Sigma(G_\Sigma \cap \Gamma)$ is a discrete subgroup which has finite index in $p_\Sigma(\Gamma)$, and it follows that $p_\Sigma(\Gamma)$ is discrete. Therefore the closure of $p_\Sigma(\Gamma)$ cannot contain $\prod_{i \in \Sigma} G_i^{(\infty)}$ since the latter is non-discrete, so \ref{it:irre1'} also fails. 

We now make the additional assumption that $\QZ(G_i)=1$ for all $i$, i.e.\ that each factor $G_i$ is just-non-compact. In particular, $G_i$ is monolithic with a non-discrete cocompact monolith by Proposition~\ref{prop:JustNonCompact}. The monolith $\Mon(G_i)$ is compactly generated (because it is cocompact), non-discrete and characteristically simple. Thus it has a trivial quasi-center and trivial polycompact radical by Proposition~\ref{prop:trivialQZ}. We may therefore apply Lemma \ref{lem-1-2}(ii), which shows that \ref{it:irre1'} $\Rightarrow$ \ref{it:irre2}. This completes the proof of statement \ref{item-thmirred-JNC}.

Finally in order to have \ref{item-thmirred-hJNC}, it remains to prove the implication \ref{it:irre2} $\Rightarrow$ \ref{it:irre4} under the extra assumption that all factors $G_i$ are hereditarily just-non-compact. Assume that  \ref{it:irre4} fails and let $\Lambda \leq \Gamma$ be a finite index subgroup of $\Gamma$ such that $\Lambda = A \times B$ with $A, B$ non-trivial. Let $H_i = \overline{p_i(\Lambda)}$ for all $i$. The index of $H_i$ in $G_i$ is bounded above by $[\Gamma: \Lambda]$, and is thus finite. In particular $H_i$ is hereditarily just-non-compact by assumption. For each $i$, the closures $\overline{p_i(A)}$ and $\overline{p_i(B)}$ are closed normal subgroup of $H_i$ that centralizer each other. If they are both non-trivial, then they both contain the monolith of $H_i$, which must then be abelian. This contradicts  Proposition~\ref{prop:JustNonCompact}(ii). Thus for each $i$, either $p_i(A) =\{1\}$ or $p_i(B) =\{1\}$. It follows that \ref{it:irre2} fails. 
\end{proof}

\begin{proof}[Proof of Corollary~\ref{cor:3factors:JNC}]
Consider the partition $\Pi_{1} \cup \dots \cup \Pi_{\ell}$ given by Corollary~\ref{cor:ReductionToIrr}. If $\ell > 1$ then one block $\Pi_{j}$ must have cardinality one. If $i$ is the unique element of $\Pi_{j}$, then the projection $p_i(\Gamma)$ is both discrete and dense in the non-discrete group $G_i$, which is absurd. So $\ell = 1$, and $\Gamma$ satisfies \ref{it:irre3}. By Theorem \ref{thm:irred:JNC} it also satisfies \ref{it:irre0} and \ref{it:irre1'}, and we are done.
\end{proof}

We present two supplements to Theorem~\ref{thm:irred:JNC} in case of groups with trivial amenable radical. 

The first one relates condition~\ref{it:irre2} to residual finiteness of the lattice, relying on  \cite{CapWes}. It requires the ambient group to have a trivial amenable radical. 

\begin{cor}\label{cor:irred:RamenTrivial}
	Let $G_1,\ldots,G_n$ be non-discrete, compactly generated,  quasi just-non-compact, tdlc groups with trivial amenable  radical. Let   $\Gamma \leq G = G_1 \times \ldots \times G_n$ be a cocompact lattice such that $p_i \colon \Gamma \rightarrow G_i$ has dense image for every $i=1,\ldots,n$. If $n \geq 4$, assume in addition that  $\Gamma$ satisfies \ref{it:irre0}.  
	If $\Gamma $ is residually finite, then it satisfies \ref{it:irre2}. 

\end{cor}

\begin{proof}
By \cite[Theorem~5.13]{CapWes}, the hypothesis that $\Gamma$ is residually finite implies that $\QZ(G)=\{1\}$. Therefore $\QZ(G_i)=\{1\}$ for all $i$, so that $G_i$ is just-non-compact. It then follows from Theorem~\ref{thm:irred:JNC} and Corollary~\ref{cor:3factors:JNC} that  \ref{it:irre2} holds. 
\end{proof}

We emphasize that  a lattice can satisfy \ref{it:irre2} without being residually finite: this is illustrate by the Burger--Mozes simple lattices constructed in \cite{BuMo2}. 

The second result enumerates other properties that are formally equivalent to \ref{it:irre2}. It uses in an essential way the Normal Subgroup Theorem due to Bader--Shalom~\cite{BaSha}. We recall that a group $\Gamma$ is called \textbf{just-infinite} if $\Gamma$ is infinite and every non-trivial normal subgroup of $\Gamma$ is of finite index.

% ANCIENNE VERSION:	
%		\begin{cor}\label{cor:irred:RamenTrivial:2}
%	Let $G_1,\ldots,G_n$ be non-discrete, compactly generated,  quasi just-non-compact, tdlc groups with trivial amenable radical.  Let   $\Gamma \leq G = G_1 \times \ldots \times G_n$ be a cocompact lattice such that $p_i \colon \Gamma \rightarrow G_i$ has dense image for every $i=1,\ldots,n$. If $n \geq 4$, assume in addition that  $\Gamma$ satisfies \ref{it:irre0}.  Then the following conditions are equivalent. 
%	\begin{enumerate}[label=(\roman*)]
%		\item $\Gamma$ satisfies  \ref{it:irre2}. 
%		
%		\item   $\QZ(G)=\{1\}$. 
%		
%		\item  $G_i$ is just-non-compact for all $i$. 
%		
%		\item $\Gamma$ is just-infinite. 
%
%	\end{enumerate}
%	
%\end{cor}
\begin{cor}\label{cor:irred:RamenTrivial:2}
	Let $G_1,\ldots,G_n$ be non-discrete, compactly generated,  quasi just-non-compact, tdlc groups. Let   $\Gamma \leq G = G_1 \times \ldots \times G_n$ be a cocompact lattice such that $p_i \colon \Gamma \rightarrow G_i$ has dense image for every $i=1,\ldots,n$. If $n \geq 4$, assume in addition that  $\Gamma$ satisfies \ref{it:irre0}.  Consider the following conditions.
\begin{enumerate}[label=(\roman*)]
		\item $\Gamma$ satisfies  \ref{it:irre2}. 
		
		\item   $\QZ(G)=\{1\}$. 
		
		\item  $G_i$ is just-non-compact for all $i$. 
		
		\item $\Gamma$ is just-infinite. 
		
\end{enumerate}

Then we have $(ii) \Leftrightarrow (iii)  \Rightarrow (i) \Leftarrow (iv)$.

If in addition $G_i$ has a trivial locally elliptic radical for all $i = 1, \dots, n$, then $(i), (ii), (iii), (iv)$ are all equivalent. 	
\end{cor}

\begin{proof}
Assume that   $\QZ(G)=\{1\}$. Then   $\QZ(G_i)=\{1\}$ for all $i$, so that $G_i$ is just-non-compact by Proposition~\ref{prop-carac-qjnc}, so (ii) $\Rightarrow$ (iii).

 Conversely, if $G_i$ is just-non-compact for all $i$, then $\QZ(G_i)=\{1\}$ by Proposition~\ref{prop:JustNonCompact}, hence $\QZ(G)=\{1\}$ by \cite[Lemma~5.5]{CapWes}, so (iii) $\Rightarrow$ (ii). 
 
 If $G_i$ is just-non-compact for all $i$, then    \ref{it:irre2} holds by Theorem~\ref{thm:irred:JNC} and Corollary~\ref{cor:3factors:JNC}.  Thus  (iii) $\Rightarrow$ (i).

Assume that $\Gamma$ does not satisfy  \ref{it:irre2}, then there exists $i$ such that $p_i \colon \Gamma \to G_i$ has a non-trivial kernel. Since the projection $p_i(\Gamma)$ is dense in $G_i$, it is an infinite group. Hence  we have found a non-trivial normal subgroup of $\Gamma$ which is of infinite index. Thus  (iv) $\Rightarrow$ (i). 

We know assume in addition that $G_i$ has a trivial locally elliptic radical for all $i = 1, \dots, n$. We must show that  (i) $\Rightarrow$ (iiii) and  (i) $\Rightarrow$ (iv). We assume henceforth that   \ref{it:irre2} holds. Then for all $i$, the projection of $\Gamma$ to $\prod_{j \neq i} G_j$ is injective. Denoting  by $H$ the closure of the projection of $\Gamma$ to $\prod_{j \neq i} G_j$, we see that $H$ is a closed cocompact subgroup of  $\prod_{j \neq i} G_j$ of finite covolume. Therefore every lattice in $H$ is a lattice in  $\prod_{j \neq i} G_j$, and thus has a trivial centralizer by \cite[Corollary~5.3]{CapWes}. We may then invoke \cite[Lemma~5.11]{CapWes}, which ensures that $\QZ(G_i)=\{1\}$. Thus (i) $\Rightarrow$ (iii).

By   \cite[Theorem~3.7(iv)]{BaSha} and \cite[Theorem~0.1]{Shalom}, we know that for every normal subgroup $N$ of $\Gamma$, the quotient $\Gamma/N$ is finite provided $G_i/\overline{p_i(N)}$ is compact for all $i$. If $N$ is non-trivial and $\Gamma$ satisfies  \ref{it:irre2}, then $G_i$ is just non-compact (as we have seen in the previous paragraph) and $\overline{p_i(N)}$ is a non-trivial closed normal subgroup, so the quotient $G_i/\overline{p_i(N)}$ is indeed compact. Thus (i) $\Rightarrow$ (iv). 
\end{proof}

\section{Covolume bounds} \label{sec:Wang}

\subsection{A Chabauty continuity property of projection maps}

Given a locally compact group $G$ and a continuous   homomorphism $\varphi\colon G \to Q$, the induced map $\varphi_* \colon \sub(G) \to \sub(Q) : H \mapsto \overline{\varphi(H)}$ need not be continuous. For example, consider $G = \RR\times \RR$ and $\varphi \colon G \to \RR$ the projection to the first factor. The sequence $H_n = n \ZZ[\sqrt 2] = \{ n(a+b\sqrt{2},a-b\sqrt{2}) \, : \, a,b\in \ZZ\}$ Chabauty converges to the trivial subgroup of $G$, but $\varphi(H_n)$ is dense in $\RR$ for all $n$. However, the map $\varphi_*$ is always semi-continuous, in the sense of the first assertion of the following proposition. Crucial for our purposes is the fact that the map $\varphi_*$ is actually continuous under an assumption of $(r, U)$-cocompactness: 

\begin{prop}\label{prop:ProjectionChabautyContinuous}
	Let $G$ and $Q$ be locally compact groups and $\varphi \colon G \to Q$ be a continuous homomorphism. Let also $(H_k)$ be a net of closed subgroups of 
	$G$ that Chabauty converges to  $H \leq G$.
	
	\begin{enumerate}[label=(\roman*)]
		\item  	Any accumulation point of the net $(\overline{\varphi(H_k)})$ in the Chabauty space $\sub(Q)$ contains $\overline{\varphi(H)}$. 
		
		\item Assume in addition  that $G$ is compactly generated tdlc and that $(H_k)$ is $(r, U)$-cocompact for some $r$, $U$ and all $k$. Then $(\overline{\varphi(H_k)})$ Chabauty converges to $\overline{\varphi(H)}$. 
	\end{enumerate}
\end{prop}

\begin{proof}
	Assertion (i) is an easy verification, and we leave the details to the reader. 
	
	\medskip \noindent 
	For (ii), let $J \leq Q$ be an accumulation point of the net $(\overline{\varphi(H_k)})$ and let   $(\overline{\varphi(H_{k'}}))$ be a subnet converging to $J$. Let $q \in J$ and let $V \leq Q$ be any compact open subgroup. For all sufficiently large $k$, the intersection $qV \cap \overline{\varphi(H_{k'})}$ is non-empty. Thus there exists $h_{k'} \in H_{k'}$ with $\varphi(h_{k'}) \in qV$. Since $O = \varphi^{-1}(V)$ is an open subgroup of $G$, it follows from Lemma~\ref{lem:rU-intersect-open} that $H_{k} \cap O$ is $(r, U\cap O)$-cocompact in $O$ for all $k$. By Lemma~\ref{lem-rU-cocom-uniform}, there exists a compact subset $L \subset O$ such that $L(H_k \cap O)  = O$ for all $k$. In particular  $g L (H_{k'} \cap O) = h_{k'} O$ for all $k$, where $g \in G$ is a fixed element with $\varphi(g) \in q V$. Hence, for all $k$, we may find $y_{k'} \in H_{k'} \cap O$ such  that $h_{k'} y_{k'} \in gL$. Since $gL$ is compact, we may assume after a further extraction that  $(h_{k'} y_{k'})_k$ converges to some limit $h$, which must  belong  to $H = \lim_k H_k$. We deduce that $\varphi(h) = \lim_k \varphi(h_{k'} y_{k'})$, which is contained in  $qV$ since $\varphi(h_{k'} y_{k'}) \in qV \varphi(O) =qV$ for all $k$. 
	
	Since $V$ was an arbitrary compact open subgroup of $Q$, we infer that $q \in \overline{\varphi(H)}$. Hence $J \leq  \overline{\varphi(H)}$. By (i) we have $J =  \overline{\varphi(H)}$. The assertion follows. 
\end{proof}

\subsection{On the set of $(r, U)$-cocompact lattices with dense projections}

\begin{thm}\label{thm:LrU-ChabClosed}
		Let $G_1,\ldots,G_n$ be non-discrete compactly generated quasi just-non-compact tdlc groups with  $n \geq 2$, and let $U \leq G = G_1 \times \dots \times G_n$ be a compact open subgroup. For every $r > 0$, the set $\mathcal L_{r, U}$ is discrete $(r, U)$-cocompact subgroups $\Gamma\leq G$ with $p_i(\Gamma)$ dense in $G_i$ for all $i$, is Chabauty closed. 
\end{thm}

We proceed in several steps. The first one is to show that the subset $\mathcal{L}^{\textrm{irr}}_{r,U}\subset \mathcal L_{r, U}$ consisting of those $\Gamma \in \mathcal L_{r, U}$ satisfying also \ref{it:irre3}, is Chabauty closed. This will be achieved in Proposition~\ref{prop:LrUirr-ChabClosed} below. 

We first record  some auxiliary results of independent interest. The first is a slight strengthening of Proposition 7.4 in \cite{GelanderLevit}. 

\begin{prop} \label{prop-chab-neigh-latt}
	Let $G$ be a compactly generated locally compact group with discrete polycompact radical $W(G)$. Let $\Gamma$ be a cocompact lattice in $G$, and $U$ a relatively compact symmetric open neighbourhood of $1$ such that $\Gamma \cap U =   W(G) \cap U = \{1\}$. Then $\Gamma$ admits a Chabauty neighbourhood consisting of cocompact lattices $\Lambda$ such that $\Lambda \cap U = \{1\}$.
\end{prop}
\begin{proof}
	First note that, since $\Gamma$ and $W(G)$ are both discrete, there does exist  a relatively compact symmetric open neighbourhood $U$ of $1$ such that $\Gamma \cap U =   W(G) \cap U = \{1\}$. 
	
	The  arguments from \cite[Proposition~7.4]{GelanderLevit} show in full generality that if $\Lambda$ is a lattice that is in a sufficiently small Chabauty neighbourhood of $\Gamma$, then $\Lambda \cap U$ is a compact subgroup of $\Gamma$ whose normalizer in $G$ is cocompact. By  Proposition~\ref{prop:CocoNormal}, we have $\Lambda \cap U \leq W(G)$, so that  $\Lambda \cap U \leq W(G) \cap U = \{1\}$. 
\end{proof}

\begin{cor}\label{cor:Lru-Chabauty-open}
	Let $G$ be a compactly generated tdlc group with discrete polycompact radical. For every compact open subgroup  $U \leq G$   all $r \geq 1$,  the set  of $(r, U)$-cocompact discrete subgroups of $G$ is Chabauty open. 	
\end{cor}

\begin{proof}
	Immediate from Lemma~\ref{lem-rU-coc-closed} and Proposition~\ref{prop-chab-neigh-latt}. 
\end{proof}

The following basic observation will be useful. 

\begin{lem}\label{lem:rU-cocom-Quotients}
	Let $G$ be a compactly generated tdlc group,   $r \geq 1$ and $ U \leq G$ be a compact open subgroup. For every continuous  homormorphism $\varphi \colon G \to Q$ such that $\varphi(G)$ is dense in $Q$, there exists a constant $r_Q$ and  a compact open subgroup $U_Q \leq Q$ such that $\overline{\varphi(H)} \in \Ccal_{r_Q,U_Q}(Q)$ for all $H \in \Ccal_{r,U}(G)$. 
\end{lem}

\begin{proof}
	By Lemma~\ref{lem-rU-cocom-uniform}, there exists a finite set $\Sigma \subset G$ such that $H \Sigma U = G$ for all $H$ in $\Ccal_{r,U}(G)$. This yields $Q = \overline{\varphi(H)\varphi(\Sigma)\varphi(U)} = \overline{\varphi(H)}\varphi(\Sigma)\varphi(U)$ for all $H$ in $\Ccal_{r,U}(G)$. Let $U_Q$ be a compact open subgroup of $Q$ containing $\varphi(U)$ (such a subgroup exists  by \cite[Lemma~3.1]{CapCMH}). By the converse statement in  Lemma~\ref{lem-rU-cocom-uniform}, there is a constant $r_Q$ such that $\overline{\varphi(H)} \in \Ccal_{r_Q,U_Q}(Q)$ for all $H \in \Ccal_{r,U}(G)$. 
\end{proof}

We emphasize that the $(r, U)$-cocompact subgroups considered in the following result are not assumed to have dense projections on each factor. 

\begin{cor}\label{cor:Irr3:Chabauty-closed}
	Let $G =  G_1 \times \dots \times G_n$ be product of compactly generated tdlc groups with discrete polycompact radical. For every compact open subgroup  $U \leq G$   all $r \geq 1$,  the set of $(r, U)$-cocompact closed subgroups   satisfying \ref{it:irre3} is Chabauty closed. 
\end{cor}
\begin{proof}
	Let $\Sigma \subsetneq \{1, \dots, n\}	$ be a non-empty proper subset. Let also $(H_k)$ be a net of $(r, U)$-cocompact closed subgroups   satisfying \ref{it:irre3} and converging to some closed subgroup $H \leq G$. Thus $H$ is $(r, U)$-cocompact by Lemma~\ref{lem-rU-coc-closed}. In particular $\overline{p_\Sigma(H_k)}$ and $\overline{p_\Sigma(H)}$ are all $(r', U')$-cocompact in $G_\Sigma$ by Lemma~\ref{lem:rU-cocom-Quotients}.  Assume that $\overline{p_\Sigma(H)}$ is discrete, so that  $\overline{p_\Sigma(H)} = p_\Sigma(H)$. By Proposition~\ref{prop:ProjectionChabautyContinuous}, this implies that every accumulation point of the net $\overline{p_\Sigma(H_k)}$ is also discrete. In view of Corollary~\ref{cor:Lru-Chabauty-open} applied to $G_\Sigma$, we infer that there is a subnet $\overline{p_\Sigma(H_{k'})}$ consisting of discrete subgroups, which contradicts the hypothesis that $H_k$ satisfies \ref{it:irre3} for all $k$. 
\end{proof}

\begin{prop}\label{prop:LrUirr-ChabClosed}
	Let $G_1,\ldots,G_n$ be non-discrete compactly generated quasi just-non-compact tdlc groups with  $n \geq 2$, and let $U  = U_1 \times \dots \times U_n \leq  G_1 \times \dots \times G_n = G$ be a compact open subgroup. For every $r > 0$, the set $\mathcal L_{r, U}^{\textrm{irr}}$ of discrete $(r, U)$-cocompact subgroups $\Gamma\leq G$ satisfying \ref{it:irre3} and with $p_i(\Gamma)$ dense in $G_i$ for all $i$, is Chabauty closed. 
\end{prop}

\begin{proof}
Let $(\Gamma_k)$ be a sequence in $\mathcal{L}^{\textrm{irr}}_{r,U}$ that converges to some closed subgroup $H \leq G$. By Lemma~\ref{lem-rU-coc-closed}, the group $H$ is $(r, U)$-cocompact. By Corollary~\ref{cor:Irr3:Chabauty-closed}, it satisfies  \ref{it:irre3}  and by Proposition~\ref{prop:ProjectionChabautyContinuous}, the projection $p_i(H) $ is dense in $G_i$ for all $i$. It remains to show that $H$ is discrete. 

Assume that this is not the case. Then $H$ satisfies all the hypothesis of Proposition~\ref{prop-H=G}, and it follows that $p_1( H \cap (G_1 \times U_2 \times \dots \times U_n))$ is a finite index open subgroup of $G_1$. The sequence $( \Gamma_k \cap (G_1 \times U_2 \times \dots \times U_n))_k $ converges to $H \cap (G_1 \times U_2 \times \dots \times U_n)$. Moreover, the restriction of $p_1$ to $G_1 \times U_2 \times \dots \times U_n$ is proper, hence it induces a Chabauty continuous map $\sub(G_1 \times U_2 \times \dots \times U_n) \to \sub(G_1)$. It follows that some finite index open subgroup of $G_1$ is a Chabauty limit of the sequence of lattices $\big(p_1(\Gamma_k \cap (G_1 \times U_2 \times \dots \times U_n))\big)_k $. From  Theorem~\ref{thmintro:Zassenhaus}, we deduce that $G_1$ has a compact open subgroup $V$ that is an infinitely generated pro-$p$ for some prime $p$.

On the other hand, the intersection $\Gamma_1\cap (V \times G_2 \times \dots \times G_n)$ is a cocompact lattice in  $V \times G_2 \times \dots \times G_n$ (by Lemma~\ref{lem:rU-intersect-open}), hence it is finitely generated since $V \times G_2 \times \dots \times G_n$ is compactly generated. Since $p_1(\Gamma_1)$ is dense in $G_1$, it follows that $p_1(\Gamma_1\cap (V \times G_2 \times \dots \times G_n))$ is dense in $V$. Thus $V$ is topologically finitely generated. This is a contradiction. 
\end{proof}

We are now ready to complete the proof of Theorem~\ref{thm:LrU-ChabClosed}. 

\begin{proof}[Proof of Theorem~\ref{thm:LrU-ChabClosed}]
For every partition $\mathcal P = \Pi_1 \cup \dots \cup \Pi_\ell$ of the set $\{1, \dots, n\}$, we denote by $\mathcal L_{r, U}(\mathcal P)$ the set of those $\Gamma \in \mathcal L_{r, U}$ such that for every $i$, the projection $p_{\Pi_i}(\Gamma)$ is discrete, and for every non-empty $\Sigma_i \subsetneq \Pi_i$, the projection $p_{\Sigma_i}(\Gamma)$ is non-discrete.

We claim that the set $\mathcal L_{r, U}(\mathcal P)$ is Chabauty closed. Indeed, let $(\Gamma_k)$ be a sequence in $\mathcal{L}_{r,U}(\mathcal P)$ that converges to some closed subgroup $H \leq G$. By Lemma~\ref{lem-rU-coc-closed}, the group $H$ is $(r, U)$-cocompact   and by Proposition~\ref{prop:ProjectionChabautyContinuous}, the projection $p_i(H) $ is dense in $G_i$ for all $i$. By Lemma~\ref{lem:rU-cocom-Quotients}, for every $i$, the projection $p_{\Pi_i}(\Gamma_k)$ is $(r_i, U_{\Pi_i})$-cocompact in $G_{\Pi_i}$. By Proposition~\ref{prop:ProjectionChabautyContinuous} and Proposition~\ref{prop:LrUirr-ChabClosed}, it follows that $p_{\Pi_i}(H)$ is discrete. Given a non-empty $\Sigma_i \subsetneq \Pi_i$, the projection $p_{\Sigma_i}(\Gamma_k)$ is non-discrete for all $k$, so the same holds for $p_{\Sigma_i}(H)$ by Corollary~\ref{cor:Lru-Chabauty-open}. Since $H$ is contained in $p_{\Pi_1}(H) \times \dots \times p_{\Pi_\ell}(H)$, it follows that $H$ is discrete, and hence that $H$ belongs to $\mathcal L_{r, U}(\mathcal P)$. 

We conclude the proof by observing that the set $\mathcal L_{r, U} $  is the union of all sets $\mathcal L_{r, U}(\mathcal P)$ where $\mathcal P$ runs over the finite set of all partitions of  $\{1, \dots, n\}$: indeed, that follows from Corollary~\ref{cor:ReductionToIrr}. Hence $\mathcal L_{r, U} $ is a finite union of closed sets, hence it is closed. 
\end{proof}

\subsection{A first Wang finiteness theorem}

The goal of this section is to prove Theorem~\ref{thm:Wang1}. We need the following observation, whose proof is inspired by Proposition 2.1 in \cite{BuMo_Wang} and Theorem~9.3 in \cite{GelanderLevit}.

\begin{prop}\label{prop:Wang:GelanderLevit}
	Let $G$ be a compactly generated tdlc group  and $\Gamma \leq G$ be a  lattice. Let $\mathcal L$ be a set of discrete subgroups of $G$ containing $\Gamma$. Assume that the following conditions hold:
	\begin{enumerate}[label=(\arabic*)]
		\item $\Gamma$ is finitely generated. 
		
		\item For every finite index subgroup $\Sigma \leq \Gamma$, the normalizer $\mathrm N_G(\Sigma)$ is discrete. 
		
		\item $\mathcal L$ is Chabauty closed. 
		
		\item For every $\Lambda \in \mathcal L$, there is an identity neighbourhood $U$ in $G$ and a neighbourhood $\Omega$ of $\Lambda$ in $\sub(G)$ such that $H \cap U = \{1\}$ for all $H \in \mathcal L \cap \Omega$. 
		
	\end{enumerate}
Then $\mathcal L$ is finite. 
\end{prop}

\begin{proof}
	Suppose for a contradiction that the set in question is infinite. In view of (3), this implies the existence of a sequence  $(\Lambda_n)_n$ of pairwise distinct discrete subgroups of $G$ all containing $\Gamma$, that Chabauty converges to some $\Lambda \in \mathcal L$.	By (4), we may find a relatively compact symmetric open neighbourhood $U$ of $1$ such that $\Lambda_n \cap U = 1$. We now argue as in the proof of Theorem 9.3 from \cite{GelanderLevit}: the existence of $U$ implies that there is a lower bound for the covolume of $\Lambda_n$ in $G$, and hence an upper bound for the index of $\Gamma$ in $\Lambda_n$. Hence in order to derive a contradiction we may assume that this index is constant. It follows that there exists a finite index normal subgroup $\Sigma_n$ in $\Lambda_n$ contained in $\Gamma$ and whose index does not depend on $n$. Since a finitely generated has finitely many subgroups of any given index,  we may extract a further subsequence so that $\Sigma_n = \Sigma$ becomes independent of $n$. 
		Condition (2) implies that $\Sigma$ is of finite index in its normalizer. In particular there are only finitely many subgroups between $\Sigma$ and its normalizer, which contradicts the fact that infinitely many $\Lambda_n$ normalize $\Sigma$. 
\end{proof}

In the proof of Theorem~\ref{thm:Wang1} in the case of Kazhdan groups we will appeal to the following result.

\begin{prop} \label{prop-kazhdan-no-tower}
Let $G_1,\ldots,G_n$ be non-discrete, compactly generated, quasi just-non-compact tdlc groups with Kazhdan's property (T), with $n>1$. Let $ \Gamma \leq G = G_1 \times \dots \times G_n$ be a non-uniform lattice such that the projection $p_i(\Gamma)$ is dense in $G_i$ for all $i$. Then the set of discrete subgroups of $G$ containing $\Gamma$ is Chabauty closed. 
\end{prop}

\begin{proof}
Let $(\Gamma_k)$ a sequence of discrete subgroups of $G$ with $\Gamma \leq \Gamma_k$ and such that $(\Gamma_k)$ converges to a closed subgroup $H$ in $\sub(G)$. Clearly $\Gamma \leq H$, and hence $p_i(H)$ is dense in $G_i$ for all $i$. We want to show that $H$ is discrete. 

Assume first that $\Gamma$ satisfies \ref{it:irre3}. Then $H$ also does. Assume for a contradiction that $H$ is not discrete. Then $H$ is a non-discrete closed subgroup of finite covolume in $G$, and $G$ has property (T), so by Proposition~\ref{prop-H=G}(i) the group $H$ is cocompact in $G$. It then follows from Lemma~\ref{lem-rU-coc-closed} that $\Gamma_k$ is cocompact in $G$ for all sufficiently large $k$, and hence that $\Gamma$ is cocompact. This is a contradiction. So we deduce that the group $H$ is discrete in this case.
	
To deal with the general case, the conclusion of Corollary~\ref{cor:ReductionToIrr} (which holds here since $G$ has Kazhdan's property (T); see Remark~\ref{rmq-(T)-3}) yields a partition $\Pi_1\cup \dots \cup \Pi_\ell$ of $\{1, \dots, n\}$ such that $p_{\Pi_i}(\Gamma)$ is a lattice with \ref{it:irre3} in $G_{\Pi_i}$ for all $i$. By Proposition~\ref{prop:ProjectionChabautyContinuous}(i), the projection $\overline{p_{\Pi_i}(H)}$ is contained in every accumulation point of the sequence  $\overline{p_{\Pi_i}(\Gamma_k)}$. Since $\Gamma$ is of finite index in $\Gamma_k$ and since $\overline{p_{\Pi_i}(\Gamma)}$ is discrete, it follows that $\overline{p_{\Pi_i}(\Gamma_k)}=  {p_{\Pi_i}(\Gamma_k)}$ is discrete. The first part of the proof then implies that $\overline{p_{\Pi_i}(H)}= p_{\Pi_i}(H)$ is discrete for all $i$. Since $H$ is contained in $\prod_{i=1}^\ell  p_{\Pi_i}(H)$, we deduce that $H$ is discrete, as required. 
\end{proof}

\begin{proof}[Proof of Theorem~\ref{thm:Wang1}]
Let $\mathcal L$ be the collection of discrete subgroups of $G$ containing $\Gamma$. 

Consider first the case where $\Gamma$ is cocompact. Let $U \leq G$ be a compact subgroup and set $r = |\Gamma \backslash G/U|$. Notice that every discrete subgroup  $\Lambda \leq G$ containing $\Gamma$ has dense projection to $G_i$ for all $i$. 
 It follows from Theorem~\ref{thm:LrU-ChabClosed}  that the set of all those $\Lambda$ is Chabauty closed. Thus Condition~(3) of Proposition~\ref{prop:Wang:GelanderLevit} is verified. Since $\Gamma$ is cocompact and $G$ is compactly generated, (1) holds as well. Condition~(2) follows from Corollary~\ref{cor:NormalizerLattice}. Finally (4) is satisfied in view of  Proposition~\ref{prop-chab-neigh-latt}. Therefore $ \mathcal L$ is finite   by Proposition~\ref{prop:Wang:GelanderLevit}. 

We now assume that $\Gamma$ is not cocompact, and that $G$ has Kazhdan's property (T). Under these assumptions, the fact that $\mathcal L$ is closed in $\sub(G)$ has been proved in Proposition \ref{prop-kazhdan-no-tower}. Since $G$ has property (T), it follows that $\Gamma$ is finitely generated, and Condition~(2) of Proposition~\ref{prop:Wang:GelanderLevit} follows from \cite[Corollary~5.4]{CapWes} since $G$ has discrete amenable radical. Finally (4) is also satisfied in view of Theorem~\ref{thm:Nbd}. Therefore by applying Proposition~\ref{prop:Wang:GelanderLevit} again, it follows that $\mathcal L$ is finite. 	
\end{proof}

\subsection{A uniform discreteness statement for irreducible lattices in products} \label{subsec-KazMar}

The following is a reformulation of Theorem~\ref{thmintro:KazMar}.

\begin{thm}\label{thm:KazMar:Irr}
Let $G_1,\ldots,G_n$ be non-discrete compactly generated quasi just-non-compact tdlc groups with  $n \geq 2$, let $U \leq G = G_1 \times \dots \times G_n$ be a compact open subgroup and let $r \geq 1$. Let also $\mathcal L_{r, U}$ be the set of discrete $(r, U)$-cocompact subgroups $\Gamma\leq G$ with $p_i(\Gamma)$ dense in $G_i$ for all $i$. Then  there exists an identity neighbourhood $V$ such that $V \cap \Gamma = \{1\}$ for every $\Gamma \in \mathcal{L}_{r,U}(G)$.  Moreover, the set of covolumes $\mathrm{covol}(\Gamma)$, for $\Gamma \in \mathcal{L}_{r,U}(G)$, is finite.
\end{thm}

\begin{proof}
By Theorem~\ref{thm:LrU-ChabClosed}, the set $\mathcal{L}_{r,U} \subset \sub(G)$ is Chabauty closed, hence compact. Moreover, by Proposition~\ref{prop-chab-neigh-latt}, for each $\Gamma \in \mathcal{L}_{r,U}$, there is a identity neighbourhood $V(\Gamma)$ in $G$ and an open Chabauty neighbourhood $\mathcal X(\Gamma)$ of $\Gamma$ such that $\Lambda \cap V(\Gamma) = \{1\}$ for all $\Lambda \in \mathcal X(\Gamma)$. Hence the collection  $\{\mathcal X(\Gamma)\}_\Gamma$ forms an open cover of the compact set $\mathcal{L}_{r,U}$. Let $\mathcal X(\Gamma_1) \cup \dots \cup \mathcal X(\Gamma_\ell)$ be a finite subcover. Then $V = \bigcap_{i=1}^\ell V(\Gamma_i)$ is an identity neighbourhood in $G$ such that $V \cap \Gamma = \{1\}$ for every $\Gamma \in \mathcal{L}_{r,U}$. 

The finiteness of the set of covolumes follows directly from the existence of $V$ together with Serre's covolume formula (see \cite[Proposition~1.4.2(b)]{Bourdon2000}). 
\end{proof}

The neighbourhood $V$ in Theorem \ref{thm:KazMar:Irr} a priori depends on $r$. As mentioned in the introduction, we do not know the answer to the following question (even in the case where all the factors are topologically simple):

\begin{question} \label{quest-uniform-neighb}
Let $G_1,\ldots,G_n$ be non-discrete compactly generated quasi just-non-compact tdlc groups, $n \geq 2$. Does there exist an identity neighbourhood $V$ in $G = G_1 \times \dots \times G_n$ such that for every cocompact lattice $\Gamma\leq G$ with $p_i(\Gamma)$ dense in $G_i$ for all $i$, we have $V \cap \Gamma = \{1\}$  ?
\end{question}

\subsection{A second Wang finiteness theorem in compactly presented groups}

Recall that a locally compact group $G$ is \textbf{compactly presented} if $G$ admits a compact generating subset $S$ and a presentation with set of generators $S$ and relators of bounded word length. When $G$ is a compactly generated tdlc group, $G$ is compactly presented if and only if $G$ acts properly and cocompactly on a coarsely simply connected locally finite graph (see \cite[Corollary~8.A.9]{CorHar}).

We will invoke the following  result, extracted from the work of Gelander--Levit \cite{GelanderLevit}. 

\begin{thm}[Gelander--Levit] \label{thm-aut(G)-orbit-open}
Let $n \geq 1$ and $G_1,\ldots,G_n$ be compactly presented tdlc groups with trivial polycompact radical. For each $i$ let $X_i$ be a connected locally finite graph on which $G_i$ acts properly and cocompactly, and let $c > 0$ such that each Rips $2$-complex $\mathrm{Rips}_c^2(X_i)$ is simply connected. Let $\mathcal{A}_i$ be the automorphism group of $\mathrm{Rips}_c^2(X_i)$, and write $\mathcal{A} = \mathcal{A}_1 \times \ldots \times \mathcal{A}_n$. Let also $\Gamma \leq G = G_1 \times \ldots \times G_n$ be a cocompact lattice. 
\begin{enumerate}[label=(\roman*)]
	\item   $\Gamma$ has a neighbourhood  in $\sub(G)$ consisting of $\mathcal{A}$-conjugates of $\Gamma$. 
	
	\item  If moreover $p_i(\Gamma)$ is dense in $G_i$ for all $i$, then  $\Gamma$ has a neighbourhood  in $\sub(G)$ consisting  of $\mathrm N_\mathcal{A}(G)$-conjugates of $\Gamma$.
\end{enumerate} 
\end{thm}

\begin{proof}
This statement is a consequence of \cite[Corollary 4.3-4.4]{GelanderLevit} (and their proofs) together with \cite[Corollary 7.1.2]{GelanderLevit}. Note that \cite[Corollary 4.4]{GelanderLevit} is stated there in the case of two factors, but the same proofs works for $n \geq 2$ factors with the assumption that the projection of $\Gamma$ on each individual factor is dense.
\end{proof}

\begin{thm} \label{thm-finite-A-orbits}
Let $n \geq 2$, and $G_1,\ldots,G_n$ be (non-discrete) compactly presented quasi just-non-compact tdlc groups with trivial polycompact radical. Let $U = U_1 \times \dots \times U_n \leq G = G_1 \times \dots \times G_n$ be a compact open subgroup. For each $i$ let $X_i$ be a connected locally finite graph on which $G_i$ acts properly and cocompactly, and let $c > 0$ such that each Rips $2$-complex $\mathrm{Rips}_c^2(X_i)$ is simply connected. Let $\mathcal{A}_i$ be the normalizer of $G_i$ in the isometry group of $\mathrm{Rips}_c^2(X_i)$, and $\mathcal{A} = \mathcal{A}_1 \times \ldots \times \mathcal{A}_n$. Then for every $r \geq 1$, the   the set $\mathcal L_{r, U}$ of $(r, U)$-cocompact lattices $\Gamma \leq G$ with $p_i(\Gamma)$ dense in $G_i$ for all $i$, is covered by finitely many $\mathcal N_\mathcal{A}(G)$-orbits.
\end{thm}

\begin{proof}
By Theorem~\ref{thm:LrU-ChabClosed}, the set 	$\mathcal L_{r, U}$  is Chabauty closed, hence compact. By  Theorem~ \ref{thm-aut(G)-orbit-open}(ii), every $\Gamma \in \mathcal L_{r, U}$ has a Chabauty neighbourhood that consists of conjugates of $\Gamma$ under $\mathrm N_\mathcal{A}(G)$. The result follows by extracting a finite subcover. 
\end{proof}

We finish by observing that Theorem~\ref{thmintro:Wang2} from the introduction follows from the more precise Theorem \ref{thm-finite-A-orbits}.

\section{Automorphism groups of graphs and local action} \label{sec:graphs}

In this section we apply the results of the previous sections to lattices in products of automorphism groups of graphs with certain local actions. Although the main example to keep in mind is that of trees, many of our results apply to arbitrary connected locally finite graphs. 

\subsection{Graph-restrictive permutation groups and restrictive actions} \label{subsec-graph-rest}

We briefly recall some notation. Let $X$ be a connected graph and let $G\leq \Aut(X)$. Given a vertex $x \in VX$, we denote by $X(x)$ the one-ball around $x$, and by $G_x^{X(x)}$ the permutation group induced by the action of $G_x$ on $X(x)$. We call $G_x^{X(x)}$ the \textbf{local action} of $G$ at $x$. Given a permutation group $L$, we say that the $G$-action on $X$ is \textbf{locally $L$} if $G_x^{X(x)}$ is permutation isomorphic to $L$ for every vertex $x$ of $X$.

Following Poto\v cnik--Spiga--Verret~\cite{PSV}, we will use the following terminology. 

\begin{defin}
A finite permutation group $L$ is \textbf{graph-restrictive} if there exists a constant $c_L$ such that, for every connected graph $X$ and every vertex-transitive discrete subgroup $G\leq \Aut(X)$ that is locally $L$, we have $G_x^{[c_L]} = 1$ for every $x \in VX$.
\end{defin}

Our formulation of that definition actually differs from that in \cite{PSV}, but both definitions are equivalent. A conjecture of R.~Weiss~\cite[Conjecture~3.12]{Weiss78} asserts that finite primitive groups are graph-restrictive; C.~Praeger~\cite[Problem~7]{Praeger98} conjectures that quasi-primitive groups are graph-restrictive, while Poto\v cnik--Spiga--Verret~\cite{PSV} conjecture that a finite transitive permutation group is graph-restrictive if and only if it is semiprimitive, and prove the \enquote{only if} implication \cite{PSV} (see Section \ref{subsec:LocalAction} for the definition of a semiprimitive permutation group). We should note that a semiregular permutation group $L$ is obviously graph-restrictive (with constant $c_L = 1$), and among intransitive permutation groups, semiregular is equivalent to graph-restrictive \cite{SV-GR-INTRAN}. Recall that a permutation group is \textbf{semiregular} if all point stabilizers are trivial. 

\medskip

An important result, that follows from the work of Trofimov--Weiss, ensures finite $2$-transitive groups are all graph-restrictive. 

\begin{thm}[{Trofimov--Weiss~\cite[Theorem~1.4]{TrofimovWeiss}}] \label{thm:TrofimovWeiss}
Every finite $2$-transitive group $L$ is graph-restrictive, with  constant $c_L\leq 6$. 		
\end{thm}

We shall now introduce a notion of a graph-restrictive pair of permutation groups. For a pair $(L_1,L_2)$ of permutation groups, we say that $G\leq \Aut(X)$ is \textbf{locally-$(L_1,L_2)$} if there exists an edge $(x,y)$ of $X$ such that $G_x^{X(x)} \simeq L_1$ and $G_y^{X(y)} \simeq L_2$. 

\begin{defin} \label{def-restric-pair}
A pair of permutation groups $(L_1,L_2)$ is \textbf{graph-restrictive} if there exists a constant $c_{L_1,L_2} = c$ such that, for every connected graph $X$ and every locally $(L_1,L_2)$ discrete subgroup $G\leq \Aut(X)$ acting on $X$ with two orbits of vertices, we have $G_x^{[c]} = 1$ for every $x \in VX$.
\end{defin}

\begin{rmq}
We note that \enquote{$L$ is graph-restrictive} is not the same as \enquote{$(L,L)$ is graph-restrictive}. We also point out that the notion defined in Definition \ref{def-restric-pair} does not coincide with the one from \cite[Definition 1.1]{MorganSpigaVerret}, as we only require the uniform bound for groups of automorphisms with \textit{two} orbits of vertices. Our definition turns out to be more appropriate for our purpose, and will provide a more robust statement hereafter.
\end{rmq}

We record an elementary observation that relates the previous notions with the other considerations of this paper. 

\begin{lem}\label{lem:GraphRestrictive->ChabClosed}
Let $X$ be a connected locally finite graph.
\begin{enumerate}[label=(\roman*)]
\item If $L$ is a graph-restrictive permutation group, then the collection of vertex-transitive discrete subgroups of $\Aut(X)$ that are locally $L$ is Chabauty-closed.
\item Similarly if the pair $(L_1,L_2)$ is graph-restrictive, then the collection of discrete subgroups of $\Aut(X)$ that are locally $(L_1,L_2)$ and with two orbits of vertices, is Chabauty-closed.
\end{enumerate}
\end{lem}	

\begin{proof}
The collection of vertex-transitive closed subgroups of $\Aut(X)$ is Chabauty closed by Lemma~\ref{lem-rU-coc-closed}. Among the vertex-transitive closed subgroups, the condition of being  locally $L$ is Chabauty closed as well, since intersecting with a fixed compact open subgroup of $\Aut(X)$ (namely a vertex-stabilizer) is a continuous map on $\sub(\Aut(X))$. If $(\Lambda_k)$ is a sequence of locally $L$ vertex-transitive discrete subgroups, then the pointwise stabilizer $\fix_{\Lambda_k}(B(x, c_L))$ of the $c_L$-ball around a vertex $x$ is trivial for all $k$, so that any accumulation point of $(\Lambda_k)$ in the Chabauty space $\sub(\Aut(X)))$ also has that property, and is thus discrete. This shows the first statement. The proof of the second statement follows along the same lines.
\end{proof}

Finally in the sequel we will use the following terminology.

\begin{defin} \label{def-loc-rest-vert-trans}
Let $G\leq \Aut(X)$ be a locally transitive group of automorphisms of $X$. If $G$ acts transitively on vertices of $X$, we say that the action of $G$ on $X$ is \textbf{restrictive} if $G$ is locally $L$ and $L$ is a graph-restrictive permutation group. If $G$ acts on $X$ with two orbits of vertices, we say that the action of $G$ on $X$ is \textbf{restrictive} if $G$ is locally $(L_1,L_2)$ and $(L_1,L_2)$ is a graph-restrictive pair of permutation groups.
\end{defin}

We observe that $G\leq \Aut(X)$ has a restrictive action on $X$ if and only if the closure of $G$ does; and that among closed subgroups of $\Aut(X)$, the property of having a restrictive action defines a clopen subset in the Chabauty space of $\Aut(X)$.

\subsection{Automorphism groups of graphs with semiprimitive local action}

A basic observation due to Burger--Mozes~\cite[Proposition~1.2.1]{BuMo1} is that a closed subgroup of the automorphism group of a connected locally finite graph   whose local action at every vertex is quasi-primitive, must be quasi just-non-compact. That result can be generalized without much effort to semiprimitive local actions (see also \cite[Chapter II.7]{Tornier-phd}).

\begin{prop} \label{prop-BM-loc-semiprim}
Let $X$ be a connected locally finite graph. Let $G \leq \mathrm{Aut}(X)$ be a closed subgroup that is locally semiprimitive, and let $N$ be a normal subgroup of $G$. Then one of the following holds:
\begin{enumerate}[label=(\roman*)]
	\item \label{item-N-free} $N$ acts freely on the edges of $X$;
	\item \label{item-N-trans-edge} $N$ acts transitively on the edges of $X$;
	\item \label{item-N-trans-V2} There exists a $G$-invariant bipartition $VX = V_1 \cup V_2$ such that $N$ acts transitively on $V_2$.
\end{enumerate}
In particular $G$ is quasi just-non-compact, and if $G$ is non-discrete then there exists a compact open subgroup $U$ of $G$ such that $G = G^{(\infty)} U$.
\end{prop}

\begin{proof}
We follow the arguments from \cite[Proposition~1.2.1]{BuMo1}. Given a normal subgroup $N$ of $G$, for every vertex $x \in VX$, the stabilizer $N_x$ is a normal subgroup of $G_x$, and thus its action  on the $1$-sphere  $X(x)$ is transitive or free.  Let $V_1(N)$ be the collection of those $x \in VX$ such that $N_x$ is transitive on  $X(x)$, and notice that $V_1(N)$ is $G$-invariant. Let $V_2(N)$ be the complement of $V_1(N)$ in $VX$. 
Notice that               $VX = V_1(N) \cup V_2(N)$ is a $G$-invariant partition of the vertex-set $VX$.

For $i =1,2$, we observe that if two adjacent vertices both belong to $V_i(N)$, then we must have $V_i(N) = VX$ since $G$ is locally transitive, hence edge-transitive,  and the graph $X$ is connected. If $V_2(N) = VX$ then the pointwise fixator in $N$ of an edge is trivial, and $N$ acts freely on edges, so \ref{item-N-free} holds. If $V_1(N) = VX$, the subgroup $N$ is transitive on the edge-set $EX$ and \ref{item-N-trans-edge} holds. The only case that remains to be analyzed is when $V_1(N)$ and $V_2(N)$ are both non-empty, and hence form a $G$-invariant bipartition of the graph $X$. In that situation, consider $x \in V_2(N)$. Then every neighbour $y$ of $x$ is in $V_1(N)$, so that $N_y$ is transitive on $X(y)$. Since $X$ is connected, it follows that the star $\{x\} \cup X(x)$ contains a representative of every $\langle N_y \mid y \in X(x) \rangle$-orbit of vertices. In particular  $N$  acts transitively on $V_2(N)$, and we obtain \ref{item-N-trans-V2}. 

It follows in particular that every closed normal subgroup of $G$ is either discrete or cocompact, i.e.\ that $G$ is quasi just-non-compact. Assume finally that $G$ is non-discrete. Since $G$ is also compactly generated  (because it acts edge-transitively on $X$), the closed normal subgroup $G^{(\infty)}$ afforded by Proposition~\ref{prop-carac-qjnc} is non-discrete, and hence falls into case \ref{item-N-trans-edge} or \ref{item-N-trans-V2}. Let $U$ be the setwise stabilizer of an edge if we are in case \ref{item-N-trans-edge}, or the stabilizer of a vertex of $V_2(G^{(\infty)})$ if we are in case \ref{item-N-trans-edge}. Then we have the equality $G = G^{(\infty)} U$.
\end{proof}

\subsection{Lattices in products of graphs with semiprimitive local action} \label{subsec-prod-graph-semiprim}

In this section we consider lattices $\Gamma \leq \Aut(X_1) \times \dots \times \Aut(X_n)$ whose action on each factor is locally semiprimitive. Before going further, we note that we are now able to complete the proof of Corollary~\ref{corintro:graphs} from the introduction. 

\begin{proof}[Proof of Corollary~\ref{corintro:graphs}]
Indeed, in view of Proposition~\ref{prop-BM-loc-semiprim} each group $G_i$ is quasi just-non-compact, and therefore the assertions \ref{it:GraphsA}, \ref{it:GraphsB} and \ref{it:GraphsC} of the corollary respectively follow from Theorems~\ref{thm:Wang1}, \ref{thmintro:KazMar} and~\ref{thmintro:Wang2}.
\end{proof}

In the rest of this section we consider lattices $\Gamma \leq \Aut(X_1) \times \dots \times \Aut(X_n)$ in a more general setting than the one of Corollary \ref{corintro:graphs}, in the sense that we no longer prescribe in advance the projections of the lattices. This is counterbalanced by the fact that we impose additional conditions on the actions on one or several of the factors. In the following statement, this additional assumption is \ref{label-action-X1-rest}. See the discussion right after the statement concerning the connection between assumptions \ref{label-action-factors-sp} and \ref{label-action-X1-rest}.

\begin{thm} \label{thm-prod-trees-IL-finite}
Let $n \geq 2$ and for each $i = 1, \ldots, n$, let $X_i$ be a connected locally finite graph such that $\Aut(X_i)$ has a discrete polycompact radical. For $r \geq 1$, let $\mathcal L_r$ be the set of discrete subgroups $\Gamma \leq \Aut(X_1) \times \dots \times \Aut(X_n)$ such that:
\begin{enumerate}[label=(\arabic*)]
	\item $\Gamma$ acts on $\prod_{i=1}^n VX_i$ with at most $r$ orbits. 
	
	\item $\Gamma$ satisfies \ref{it:irre0}. 
	
	\item \label{label-action-factors-sp} For all $i = 1, \dots, n$, the action of $\Gamma$ on $X_i$ is locally semiprimitive. 
	
	\item \label{label-action-X1-rest} The action of $\Gamma$ on $X_1$ is restrictive. 

\end{enumerate}
Then:
\begin{enumerate}[label=(\roman*)]
	\item There is some $c>0$ such that $\Gamma_x^{[c]} =\{1\}$ for every vertex $x \in \prod_{i=1}^n VX_i$ and every $\Gamma \in \mathcal L_r$. In particular $\mathcal L_r$ is Chabauty closed.
	
	\item If in addition  $\Aut(X_i)$ is compactly presented with trivial polycompact radical for all $i$, then the elements of $\mathcal L_r$ fall into finitely many isomorphism classes.
\end{enumerate}
\end{thm}

Before giving the proof, we make a few comments:

\begin{enumerate}[label=\arabic*)]

\item When specified to the case of trees, Theorem~\ref{thm-prod-trees-IL-finite} provides a  partial solution to a conjecture due to Y.~Glasner  \cite[Conjecture~1.5]{Glasner-2dim}. 

\item By definition when $\Gamma$ acts vertex-transitively on $X_1$, assumption \ref{label-action-X1-rest} means that $\Gamma$ is locally graph-restrictive. As mentioned in \S \ref{subsec-graph-rest}, it is conjectured that this is always the case when the action is locally semiprimitive (so conjecturally, assumption \ref{label-action-factors-sp} automatically implies assumption \ref{label-action-X1-rest}). This has been confirmed for numerous specific classes of local actions, including $2$-transitive groups \cite{TrofimovWeiss}, primitive groups of affine type \cite{Weiss-affinetype,Spiga-affinetype}, primitive groups of twisted wreath type \cite{Spiga-wrtype}, or semiprimitive groups with two distinct minimal normal subgroups \cite{Giud-Mo-semi-prim}.
\end{enumerate}

\begin{proof}[Proof of Theorem~\ref{thm-prod-trees-IL-finite}]
Let $(\Gamma_k)$ be a sequence of elements of $\mathcal L_r$ that converges to $H$. Let us show that $H$ belongs to $\mathcal L_r$. 
	
First observe that the assumption implies that each $\Aut(X_i)$ acts on $X_i$ with finitely many orbits of vertices, and is therefore compactly generated since $X_i$ is connected. Since moreover the full automorphism group $\Aut(X_i)$ has a discrete polycompact radical, Proposition~\ref{prop-5-3} ensures that $\Gamma_k$ satisfies \ref{it:irre3} for all $k$. So by Corollary~\ref{cor:Irr3:Chabauty-closed} the group $H$ also satisfies \ref{it:irre3} (and hence \ref{it:irre0}), and acts on $\prod_{i=1}^n VX_i$ with at most $r$ orbits.
	
Proposition~\ref{prop:ProjectionChabautyContinuous} ensures that $H$ is locally semiprimitive on $X_i$ for all $i$, and restrictive on $X_1$. So it follows from Proposition~\ref{prop-BM-loc-semiprim} that the group $G_i = \overline{p_i(H)}$ is quasi just-non-compact for all $i$, and $G_i$ is non-discrete by \ref{it:irre3}.
	
So in oder to show that $H$ belongs to $\mathcal L_r$, we have to show that $H$ is discrete. We argue by contradiction and assume that $H$ is not discrete. Then by Proposition~\ref{prop-H=G} we have $H \cap G_i \geq G_i^{(\infty)}$. Moreover for all $i$, Proposition~\ref{prop-BM-loc-semiprim} provides a compact open subgroup $U_i$ of $G_i$ such that $(H\cap G_i)U_i = G_i$. In view of Proposition~\ref{prop:OpenProj}(ii), we deduce that $p_1(H \cap (G_1 \times U_2 \times \dots \times U_n)) = G_1$. It follows that $G_1$ is the Chabauty limit of the sequence of cocompact lattices $\Lambda_k := p_1(\Gamma_k \cap (G_1 \times U_2 \times \dots \times U_n)) $. We have seen above that the action of $G_1 = \overline{p_1(H)}$ on $X_1$ is restrictive, so the same is true for $\Lambda_k$ for $k$ large enough as having a restrictive action is a Chabauty open condition. But Lemma~\ref{lem:GraphRestrictive->ChabClosed} then implies that $G_1$ must be discrete. Therefore we have reached a contradiction, and hence the subgroup $H$ must be discrete. 
			
We have thus shown that $\mathcal L_r$ is Chabauty closed, and assertion~(i) follows from a compactness argument using Proposition~\ref{prop-chab-neigh-latt}, as in the proof of Theorem \ref{thm:KazMar:Irr}.
	
If in addition all the factors $\Aut(X_i)$ are compactly presented with trivial polycompact radical, by Theorem\ref{thm-aut(G)-orbit-open}(i) every cocompact lattice $\Gamma$ in $\Aut(X_1) \times \dots \times \Aut(X_n)$ has a Chabauty neighbourhood that consists of cocompact lattices isomorphic to $\Gamma$.  Thus the elements of $\mathcal L_r$ fall into finitely many isomorphism classes by the compactness of $\mathcal L_r$. 
\end{proof}

We finally establish a statement analogous to Theorem~\ref{thm-prod-trees-IL-finite}, but without any irreducibility assumption. We shall use the following auxiliary fact. 

\begin{lem}\label{lem:KazMar:aux}
	Let  $n \geq 1$ and $G_1,\ldots,G_n$ be non-discrete compactly generated tdlc groups with discrete polycompact radical. Let also $\mathcal L$ be a set of cocompact lattices in $G$. Suppose that for every non-empty subset $\Pi \subseteq  \{1, \dots, n\}$, there exists a compact open subgroup $V_\Pi \leq G_\Pi$ such that for all $\Gamma \in \mathcal L$, if $p_\Pi(\Gamma)$ is discrete and $p_{\Delta}(\Gamma)$ is non-discrete for each non-empty proper subset $\Delta \subsetneq \Pi$, then $p_\Pi(\Gamma) \cap V_\Pi = \{1\}$. Then there exists a compact open subgroup $W \leq G$ such that $\Gamma \cap W = \{1\}$ for all $\Gamma \in \mathcal L$. 
\end{lem}
\begin{proof}
	We work by induction on $n$. For $n=1$ there is nothing to prove; we assume henceforth that $n>1$. 
	
	Let $\Pi \subseteq  \{1, \dots, n\}$ be non-empty. Set 
	$$\mathcal L_\Pi = \{p_\Pi(\Gamma) \mid \Gamma \in \mathcal L \text{ and } p_\Pi(\Gamma) \text{ is discrete}\}.$$
	If $\Pi \neq  \{1, \dots, n\}$, then by the induction hypothesis, there exists a compact open subgroup $W_\Pi \leq G_\Pi$ such that  $\Lambda \cap W_\Pi = \{1\}$ for all $\Lambda \in \mathcal L_\Pi$. If  $\Pi =  \{1, \dots, n\}$, then we set $W_\Pi =V_\Pi$, where $V_\Pi$ is the compact open subgroup given by hypothesis. 
	
	Now, for all $i$, we set $W_i = \bigcap_{\Pi \ni i} W_\Pi \cap G_i$. Thus $W_i$ is a compact open subgroup of $G_i$, and for every non-empty subset $\Pi \subseteq  \{1, \dots, n\}$,  we have $\prod_{i \in \Pi} W_i \leq W_\Pi$. Finally, we set $W  = \prod_{i =1}^n W_i $. 
	
	Let $\Gamma \in \mathcal L$. If for every non-empty proper subset $\Pi \subset  \{1, \dots, n\}$, the projection $p_\Pi(\Gamma)$ is non-discrete, then we have 
	$\Gamma \cap W_{\{1, \dots, n\}} =  \Gamma \cap V_{\{1, \dots, n\}} = \{1\}$ by hypothesis, so that $\Gamma \cap W = \{1\}$. On the other hand, if there exists a non-empty  proper subset $\Pi \subset  \{1, \dots, n\}$ such that the projection $\Gamma_{\Pi} = p_\Pi(\Gamma)$ is discrete, then $\Gamma_\Sigma = p_\Sigma(\Gamma)$ is also discrete by  by Proposition~\ref{prop-5-3}. Moreover we have $\Gamma_{\Pi}\cap W_\Pi = \{1\}$ and $\Gamma_\Sigma  \cap W_\Sigma = \{1\}$. Since $\Gamma$ embeds as a subgroup of $\Gamma_{\Pi}\times \Gamma_{\Sigma}$, we have $\Gamma \cap (W_\Pi \times W_{\Sigma})=\{1\}$. In particular, we have $\Gamma \cap W=\{1\}$. 
\end{proof}

\begin{cor}\label{cor:TreesFinite}
Let $n \geq 1$ and for each $i = 1, \ldots, n$, let $X_i$ be a connected  locally finite graph such that $\Aut(X_i)$ has a discrete polycompact radical. For $r > 0$, let $\mathcal L_r$ be the set of discrete subgroups $\Gamma \leq \Aut(X_1) \times \dots \times \Aut(X_n)$ acting with at most $r$ orbits on $\prod_{i=1}^n VX_i$ and whose action on $X_i$ is restrictive and locally semiprimitive for all $i$. Then:
\begin{enumerate}[label=(\roman*)]
	\item There is some $c>0$ such that  $\Gamma_x^{[c]} =\{1\}$ for every vertex $x \in \prod_{i=1}^n VX_i$ and every $\Gamma \in \mathcal L_r$. In particular $\mathcal L_r$ is Chabauty closed. 
	
	\item If in addition  $\Aut(X_i)$ is compactly presented with trivial polycompact radical for all $i$, then the elements of $\mathcal L_r$ fall into finitely many isomorphism classes.
\end{enumerate}
\end{cor}

\begin{proof}
According to Theorem~\ref{thm-prod-trees-IL-finite}(i), the hypotheses of Lemma~\ref{lem:KazMar:aux} are satisfied by the family of lattices $\mathcal L_r$. Thus by the lemma there is a constant $c$ such that $\Gamma_x^{[c]} =\{1\}$ for some vertex $x \in \prod_{i=1}^n VX_i$. Since every $\Gamma \in \mathcal L_r$ has at most $r$ orbits of vertices, it follows from Lemma~\ref{lem-rU-cocom-uniform} that $\Gamma_x^{[c+r]} =\{1\}$ for every vertex $x \in \prod_{i=1}^n VX_i$.  Using Lemma~\ref{lem-rU-coc-closed} and Proposition~\ref{prop:ProjectionChabautyContinuous}, we deduce that $\mathcal L_r$ is closed, so that (i) holds. Assertion~(ii) follows from (i) together with Theorem~\ref{thm-aut(G)-orbit-open}.
\end{proof}

\begin{proof}[Proof of Theorem~\ref{thm-intro-prod-trees-2trans}]
Let $\mathcal L$ be the collection of lattices as in the statement. By Theorem~\ref{thm:TrofimovWeiss}, every finite $2$-transitive group is graph-restrictive, so $\Gamma$ has a restrictive action on each $T_i$ for every $\Gamma \in \mathcal L$. Since $\Aut(T_i)$ is vertex-transitive, its only compact normal subgroup is the trivial one. Hence by Corollary~\ref{cor:TreesFinite}(i), $\mathcal L$ is Chabauty closed.

To conclude, we note that $\Aut(T_i)$ is compactly presented since $T_i$ is simply connected (see \cite[Corollary~8.A.9]{CorHar}). Thus the lattices in $\mathcal L$ fall into finitely many conjugacy classes by invoking Theorem~\ref{thm-aut(G)-orbit-open} with $G_i = \mathcal A_i = \Aut(T_i)$ for all $i$. 
\end{proof}

\appendix

\section{Chabauty deformations of cofinite subgroups in Kazhdan groups}

The following interesting fact follows from results of J.~Fell~\cite{Fell1964} on continuity properties of induction of unitary representations. An analogous result has been established, with a similar argument, by S.P.\ Wang in \cite[Theorem 3.10]{SP-Wang-75}.

\begin{thm}\label{thm:Fell}
	Let $G$ be a first countable locally compact group with Kazhdan's property (T). Then the set of closed subgroups of finite covolume forms an open subset of the Chabauty space $\sub(G)$. 
\end{thm}

\begin{proof}
	Since $G$ has (T), it is compactly generated and thus second countable. This condition is needed to invoke Fell's results from \cite{Fell1964}. In that paper, the author introduces a topology on the set $\mathscr S(G)$ of pairs $(H, T)$ where $H$ is a closed subgroup of $G$ and $T$ is an   equivalence class of  unitary representations of $H$. It follows  from \cite[Lemma~3.3 and Theorem~4.2]{Fell1964} that the map $(H, T) \to \mathrm{Ind}_H^G(T)$ is continuous, where the target space is  endowed with Fell's topology. 
	
	Let $(H_n)$ be a sequence of closed subgroups of $G$ converging to $H \in \sub(G)$. For a closed subgroup $J \leq G$, let $\mathbf 1_{J}$ be the trivial representation of $J$. By \cite[Lemma~3.2]{Fell1964}, the sequence $(H_n, \mathbf 1_{H_n})$ converges to $(H, \mathbf 1_H)$ in the space $\mathscr S(G)$. Therefore, by Fell's theorem mentioned above, the sequence  $\mathrm{Ind}_{H_n}^G(\mathbf 1_{H_n})$ converges to $\mathrm{Ind}_{H}^G(\mathbf 1_{H})$ in Fell's topology. In particular $\mathrm{Ind}_{H}^G(\mathbf 1_{H})$ is weakly contained in the direct sum $\bigoplus_n \mathrm{Ind}_{H_n}^G(\mathbf 1_{H_n})$ (see \cite[Proposition~F.2.2]{BHV}).

	Assume now that $H$ is of finite covolume and suppose for a contradiction that the closed subgroups of finite covolume of $G$ do not form a neighbourhood of $H$. Then there exists a sequence $(H_n)$ converging to $H$, such that $H_n$ is not of finite covolume in $G$ for any $n$.  Since $H$ is of finite covolume,  the trivial representation $\mathbf 1_G$ is contained in $\mathrm{Ind}_{H}^G(\mathbf 1_{H})$. Since weak containment of unitary representations is a transitive relation, we deduce from the previous paragraph that $\mathbf 1_G$ is weakly contained in $\bigoplus_n \mathrm{Ind}_{H_n}^G(\mathbf 1_{H_n})$. Since $G$ has (T) by hypothesis, the trivial representation is actually contained in $\bigoplus_n \mathrm{Ind}_{H_n}^G(\mathbf 1_{H_n})$. Hence there exists $n$ such that $\mathbf 1_G$ is contained in $ \mathrm{Ind}_{H_n}^G(\mathbf 1_{H_n})$. By \cite[Theorem~E.3.1]{BHV}, this implies that $H_n$ is of finite covolume, a contradiction.
\end{proof}

\begin{rmq}
It should be noted that Theorem \ref{thm:Fell} fails without the property (T) assumption. Indeed as soon as a group $G$ has an infinitely generated lattice $\Gamma \leq G$, then there does not exist any Chabauty neighbourhood of $\Gamma$ consisting of subgroups of finite covolume in $G$. Indeed we may write $\Gamma$ as the Chabauty limit of its finitely generated subgroups, and none of them is a lattice in $G$ (since they are of infinite index in $\Gamma$ as $\Gamma$ is infinitely generated). As a concrete example, one can take the full automorphism group of a regular tree $T$, all of whose non-uniform lattices are infinitely generated, see \cite{Bass-Lubotzky}. 

We also point out that a locally compact group $G$ is compactly generated if and only if the set of its closed cocompact subgroups is open in the Chabauty space $\sub(G)$, see \cite[Remark~3.7]{GelanderLevit}.
\end{rmq}

We can now complete the proof of Theorem~\ref{thm:Nbd}, using ideas similar to those of Gelander--Levit \cite[Proposition~7.4]{GelanderLevit}.

\begin{proof}[Proof of Theorem~\ref{thm:Nbd}]
	Since $G$ is compactly generated, every identity neighbourhood $V$ contains a compact normal subgroup $K_V$ such that $G/K_V$ is second countable (see \cite{KK}). Choosing $V$ with $V \cap R(G) = \{1\}$ we deduce $K_V=1$ and $G$ is second countable. 
	
	Let now $\Gamma \leq G$ be a lattice, and $\Omega \subset \sub(G)$ be a neighbourhood of $\Gamma$ consisting of closed subgroups of finite covolume, as afforded by Theorem~\ref{thm:Fell}.  Let $U \subset G$ be a compact neighbourhood such that $U \cap \Gamma \cap R(G) = \{1\}$. By \cite[Lemma~7.3]{GelanderLevit}, there exists a neighbourhood $\Omega_U$ of $\Gamma$ contained in $\Omega$ such that for all $H \in \Omega_U$, the intersection $H \cap U$ is a subgroup that is normalized by some subgroup $L \leq H$ with $L \in \Omega$. 
	
	Notice that the group $H \cap U$ is compact, and that  $L$ is of finite covolume in $G$. Therefore $(H \cap U)L$ is a closed subgroup of finite covolume in $G$, containing $H \cap U$ as a compact normal subgroup. By the main result of \cite{BDL_URS}, it follows that $H \cap U$ is contained in the amenable radical $R(G)$. Since $U \cap R(G) =\{1\}$, it follows that $H \cap U =\{1\}$. Thus $\Omega_U$ is a neighbourhood of $\Gamma$ as required. 
\end{proof}

{\small 
\bibliographystyle{abbrv}
\bibliography{n-facteurs}

\begin{thebibliography}{10}

\bibitem{Aschbacher}
M.~Aschbacher.
\newblock {\em Finite group theory}, volume~10 of {\em Cambridge Studies in
  Advanced Mathematics}.
\newblock Cambridge University Press, Cambridge, second edition, 2000.

\bibitem{BDL_URS}
U.~Bader, B.~Duchesne, and J.~L\'ecureux.
\newblock Amenable invariant random subgroups.
\newblock {\em Israel J. Math.}, 213(1):399--422, 2016.
\newblock With an appendix by Phillip Wesolek.

\bibitem{BFS-adelic}
U.~Bader, A.~Furman, and R.~Sauer.
\newblock An adelic arithmeticity theorem for lattices in products.
\newblock {\em Math. Z.}, 2019.
\newblock To appear.

\bibitem{BaSha}
U.~Bader and Y.~Shalom.
\newblock Factor and normal subgroup theorems for lattices in products of
  groups.
\newblock {\em Invent. Math.}, 163(2):415--454, 2006.

\bibitem{BEW}
Y.~Barnea, M.~Ershov, and T.~Weigel.
\newblock Abstract commensurators of profinite groups.
\newblock {\em Trans. Amer. Math. Soc.}, 363(10):5381--5417, 2011.

\bibitem{BK90}
H.~Bass and R.~Kulkarni.
\newblock Uniform tree lattices.
\newblock {\em J. Amer. Math. Soc.}, 3(4):843--902, 1990.

\bibitem{Bass-Lubotzky}
H.~Bass and A.~Lubotzky.
\newblock {\em Tree lattices}, volume 176 of {\em Progress in Mathematics}.
\newblock Birkh\"auser Boston, Inc., Boston, MA, 2001.
\newblock With appendices by Bass, L. Carbone, Lubotzky, G. Rosenberg and J.
  Tits.

\bibitem{BHV}
B.~Bekka, P.~de~la Harpe, and A.~Valette.
\newblock {\em Kazhdan's property ({T})}, volume~11 of {\em New Mathematical
  Monographs}.
\newblock Cambridge University Press, Cambridge, 2008.

\bibitem{Bourdon2000}
M.~Bourdon.
\newblock Sur les immeubles fuchsiens et leur type de quasi-isom\'etrie.
\newblock {\em Ergodic Theory Dynam. Systems}, 20(2):343--364, 2000.

\bibitem{BuMo1}
M.~Burger and S.~Mozes.
\newblock Groups acting on trees: from local to global structure.
\newblock {\em Inst. Hautes \'Etudes Sci. Publ. Math.}, (92):113--150 (2001),
  2000.

\bibitem{BuMo2}
M.~Burger and S.~Mozes.
\newblock Lattices in product of trees.
\newblock {\em Inst. Hautes \'Etudes Sci. Publ. Math.}, (92):151--194 (2001),
  2000.

\bibitem{BuMo_Wang}
M.~Burger and S.~Mozes.
\newblock Lattices in products of trees and a theorem of {H}. {C}. {W}ang.
\newblock {\em Bull. Lond. Math. Soc.}, 46(6):1126--1132, 2014.

\bibitem{CapCMH}
P.-E. Caprace.
\newblock Amenable groups and {H}adamard spaces with a totally disconnected
  isometry group.
\newblock {\em Comment. Math. Helv.}, 84(2):437--455, 2009.

\bibitem{CaMo-discrete}
P.-E. Caprace and N.~Monod.
\newblock Isometry groups of non-positively curved spaces: discrete subgroups.
\newblock {\em J. Topol.}, 2(4):701--746, 2009.

\bibitem{CaMo-decomp}
P.-E. Caprace and N.~Monod.
\newblock Decomposing locally compact groups into simple pieces.
\newblock {\em Math. Proc. Cambridge Philos. Soc.}, 150(1):97--128, 2011.

\bibitem{CaMoKM}
P.-E. Caprace and N.~Monod.
\newblock A lattice in more than two {K}ac-{M}oody groups is arithmetic.
\newblock {\em Israel J. Math.}, 190:413--444, 2012.

\bibitem{Cap-Rad-chab}
P.-E. Caprace and N.~Radu.
\newblock Chabauty limits of simple groups acting on trees.
\newblock {\em J. Inst. Math. Jussieu}, 2019.
\newblock To appear.

\bibitem{CRW_dense}
P.-E. Caprace, C.~Reid, and P.~Wesolek.
\newblock Approximating simple locally compact groups by their dense locally
  compact subgroups.
\newblock {\em Int. Math. Res. Not.}, 2019.
\newblock To appear.

\bibitem{CRW-part2}
P.-E. Caprace, C.~D. Reid, and G.~A. Willis.
\newblock Locally normal subgroups of totally disconnected groups. {P}art {II}:
  compactly generated simple groups.
\newblock {\em Forum Math. Sigma}, 5:e12, 89, 2017.

\bibitem{CapWes}
P.-E. Caprace and P.~Wesolek.
\newblock Indicability, residual finiteness, and simple subquotients of groups
  acting on trees.
\newblock {\em Geom. Topol.}, 22(7):4163--4204, 2018.

\bibitem{Chabauty-topo}
C.~Chabauty.
\newblock Limite d'ensembles et g\'eom\'etrie des nombres.
\newblock {\em Bull. Soc. Math. France}, 78:143--151, 1950.

\bibitem{Cor-comm-focal}
Y.~Cornulier.
\newblock Commability and focal locally compact groups.
\newblock {\em Indiana Univ. Math. J.}, 64(1):115--150, 2015.

\bibitem{CorHar}
Y.~Cornulier and P.~de~la Harpe.
\newblock {\em Metric geometry of locally compact groups}, volume~25 of {\em
  EMS Tracts in Mathematics}.
\newblock European Mathematical Society (EMS), Z\"urich, 2016.
\newblock Winner of the 2016 EMS Monograph Award.

\bibitem{anprop}
J.~D. Dixon, M.~P.~F. du~Sautoy, A.~Mann, and D.~Segal.
\newblock {\em Analytic pro-{$p$}-groups}, volume 157 of {\em London
  Mathematical Society Lecture Note Series}.
\newblock Cambridge University Press, Cambridge, 1991.

\bibitem{Fell1964}
J.~M.~G. Fell.
\newblock Weak containment and induced representations of groups. {II}.
\newblock {\em Trans. Amer. Math. Soc.}, 110:424--447, 1964.

\bibitem{Furman-min-strg}
A.~Furman.
\newblock On minimal strongly proximal actions of locally compact groups.
\newblock {\em Israel J. Math.}, 136:173--187, 2003.

\bibitem{Furst-bourb}
H.~Furstenberg.
\newblock Rigidity and cocycles for ergodic actions of semisimple {L}ie groups
  (after {G}. {A}. {M}argulis and {R}. {Z}immer).
\newblock In {\em Bourbaki {S}eminar, {V}ol. 1979/80}, volume 842 of {\em
  Lecture Notes in Math.}, pages 273--292. Springer, Berlin-New York, 1981.

\bibitem{GartsideSmith}
P.~Gartside and M.~Smith.
\newblock Counting the closed subgroups of profinite groups.
\newblock {\em J. Group Theory}, 13(1):41--61, 2010.

\bibitem{Gel-homotopy}
T.~Gelander.
\newblock Homotopy type and volume of locally symmetric manifolds.
\newblock {\em Duke Math. J.}, 124(3):459--515, 2004.

\bibitem{Gel-irs}
T.~Gelander.
\newblock A lecture on invariant random subgroups.
\newblock In {\em New directions in locally compact groups}, London
  Mathematical Society Lecture Note Series, pages 186--204. Cambridge
  University Press, 2018.

\bibitem{GelanderLevit}
T.~Gelander and A.~Levit.
\newblock Local rigidity of uniform lattices.
\newblock {\em Comment. Math. Helv.}, 93(4):781--827, 2018.

\bibitem{Giud-Mo-semi-prim}
M.~Giudici and L.~Morgan.
\newblock A theory of semiprimitive groups.
\newblock {\em J. Algebra}, 503:146--185, 2018.

\bibitem{Glasner-2dim}
Y.~Glasner.
\newblock A two-dimensional version of the {G}oldschmidt-{S}ims conjecture.
\newblock {\em J. Algebra}, 269(2):381--401, 2003.

\bibitem{KK}
S.~Kakutani and K.~Kodaira.
\newblock \"{U}ber das {H}aarsche {M}ass in der lokal bikompakten {G}ruppe.
\newblock {\em Proc. Imp. Acad. Tokyo}, 20:444--450, 1944.

\bibitem{Kaz-Marg-thm}
D.~A. Ka\v{z}dan and G.~A. Margulis.
\newblock A proof of {S}elberg's hypothesis.
\newblock {\em Mat. Sb. (N.S.)}, 75 (117):163--168, 1968.

\bibitem{Kuran}
M.~Kuranishi.
\newblock On everywhere dense imbedding of free groups in {L}ie groups.
\newblock {\em Nagoya Math. J.}, 2:63--71, 1951.

\bibitem{Margulis}
G.~A. Margulis.
\newblock {\em Discrete subgroups of semisimple {L}ie groups}, volume~17 of
  {\em Ergebnisse der Mathematik und ihrer Grenzgebiete (3) [Results in
  Mathematics and Related Areas (3)]}.
\newblock Springer-Verlag, Berlin, 1991.

\bibitem{MorganSpigaVerret}
L.~Morgan, P.~Spiga, and G.~Verret.
\newblock On the order of {B}orel subgroups of group amalgams and an
  application to locally-transitive graphs.
\newblock {\em J. Algebra}, 434:138--152, 2015.

\bibitem{PSV}
P.~Poto\v{c}nik, P.~Spiga, and G.~Verret.
\newblock On graph-restrictive permutation groups.
\newblock {\em J. Combin. Theory Ser. B}, 102(3):820--831, 2012.

\bibitem{Praeger98}
C.~E. Praeger.
\newblock Finite quasiprimitive group actions on graphs and designs.
\newblock In {\em Groups---{K}orea '98 ({P}usan)}, pages 319--331. de Gruyter,
  Berlin, 2000.

\bibitem{Radu-groups}
N.~Radu.
\newblock A classification theorem for boundary 2-transitive automorphism
  groups of trees.
\newblock {\em Invent. Math.}, 209(1):1--60, 2017.

\bibitem{Raghu}
M.~S. Raghunathan.
\newblock {\em Discrete subgroups of {L}ie groups}.
\newblock Springer-Verlag, New York-Heidelberg, 1972.
\newblock Ergebnisse der Mathematik und ihrer Grenzgebiete, Band 68.

\bibitem{ReidDistal}
C.~D. {Reid}.
\newblock Distal actions on coset spaces in totally disconnected, locally
  compact groups.
\newblock {\em Journal of Topology and Analysis}, 2018.
\newblock To appear.

\bibitem{Shalom}
Y.~Shalom.
\newblock Rigidity of commensurators and irreducible lattices.
\newblock {\em Invent. Math.}, 141(1):1--54, 2000.

\bibitem{Spiga-wrtype}
P.~Spiga.
\newblock On {$G$}-locally primitive graphs of locally twisted wreath type and
  a conjecture of {W}eiss.
\newblock {\em J. Combin. Theory Ser. A}, 118(8):2257--2260, 2011.

\bibitem{Spiga-affinetype}
P.~Spiga.
\newblock An application of the local {$C(G,T)$} theorem to a conjecture of
  {W}eiss.
\newblock {\em Bull. Lond. Math. Soc.}, 48(1):12--18, 2016.

\bibitem{SV-GR-INTRAN}
P.~Spiga and G.~Verret.
\newblock On intransitive graph-restrictive permutation groups.
\newblock {\em J. Algebraic Combin.}, 40(1):179--185, 2014.

\bibitem{Tits-depl-born}
J.~Tits.
\newblock Automorphismes \`a d\'eplacement born\'e des groupes de {L}ie.
\newblock {\em Topology}, 3(suppl. 1):97--107, 1964.

\bibitem{Tits_arbre}
J.~Tits.
\newblock Sur le groupe des automorphismes d'un arbre.
\newblock pages 188--211, 1970.

\bibitem{Tornier-phd}
S.~Tornier.
\newblock {\em Groups acting on trees and contributions to {W}illis theory}.
\newblock 2018.
\newblock Ph{D} Thesis.

\bibitem{TrofimovBG}
V.~I. Trofimov.
\newblock Groups of automorphisms of graphs as topological groups.
\newblock {\em Mat. Zametki}, 38(3):378--385, 476, 1985.

\bibitem{TrofimovWeiss}
V.~I. Trofimov and R.~M. Weiss.
\newblock Graphs with a locally linear group of automorphisms.
\newblock {\em Math. Proc. Cambridge Philos. Soc.}, 118(2):191--206, 1995.

\bibitem{Wang}
H.~C. Wang.
\newblock On a maximality property of discrete subgroups with fundamental
  domain of finite measure.
\newblock {\em Amer. J. Math.}, 89:124--132, 1967.

\bibitem{Wang72}
H.~C. Wang.
\newblock {\em Topics on totally discontinuous groups}, pages 459--487. Pure
  and Appl. Math., Vol. 8.
\newblock Dekker, New York, 1972.

\bibitem{SP-Wang-75}
S.~P. Wang.
\newblock On isolated points in the dual spaces of locally compact groups.
\newblock {\em Math. Ann.}, 218(1):19--34, 1975.

\bibitem{Weiss-affinetype}
R.~Weiss.
\newblock An application of {$p$}-factorization methods to symmetric graphs.
\newblock {\em Math. Proc. Cambridge Philos. Soc.}, 85(1):43--48, 1979.

\bibitem{Weiss78}
R.~Weiss.
\newblock {$s$}-transitive graphs.
\newblock In {\em Algebraic methods in graph theory, {V}ol. {I}, {II}
  ({S}zeged, 1978)}, volume~25 of {\em Colloq. Math. Soc. J\'anos Bolyai},
  pages 827--847. North-Holland, Amsterdam-New York, 1981.

\bibitem{Zassenhaus}
H.~{Zassenhaus}.
\newblock {Beweis eines Satzes \"uber diskrete Gruppen.}
\newblock {\em {Abh. Math. Semin. Univ. Hamb.}}, 12:289--312, 1938.

\end{thebibliography}
}
\end{document}